\documentclass[a4paper]{amsart}

\usepackage{amsfonts,amsmath,amssymb,amscd,amsthm}
\usepackage{color,stmaryrd,latexsym}
\usepackage{hyperref}
\usepackage[all]{xy} 
\usepackage[T1]{fontenc}

\def\Hom{\mathop{\rm Hom}\nolimits}
\def\vol{\mathop{vol}\nolimits}
\def\ker{\mathop{\rm ker}\nolimits}
\def\im{\mathop{\rm im}\nolimits}

\def\tr{\mathop{\rm Tr}\nolimits}

\def\diff{\mbox{\sl Diff}}

\def\note#1{\marginpar{\raggedright\if@twoside\ifodd\c@page\raggedleft\fi\fi\sf\scriptsize RMK: #1}}
\def\epsilon{{\varepsilon}}
\def\ep{{\varepsilon}}

\newcommand{\SU}{\mr{SU}}
\newcommand{\SO}{\mr{SO}}
\newcommand{\Sp}{\mr{Sp}}
\newcommand{\GL}{\mr{GL}}

\newcommand{\Gt}{\mr{G_2}}
\newcommand{\spin}{\mr{Spin}}

\newcommand{\mf}{\mathfrak}
\newcommand{\mr}{\mathrm}
\newcommand{\mb}{\mathbf}
\newcommand{\mc}{\mathcal}

\newcommand{\R}{\mathbb{R}}

\newcommand{\C}{\mathbb{C}}

\newcommand{\Oc}{\mathbb{O}}

\renewcommand{\iff}{if and only if }
\newcommand{\wrt}{with respect to }

\newcommand{\Id}{\mr{Id}}

\newcommand{\be}{\begin{equation}}
\newcommand{\ben}{\begin{equation}\nonumber}
\newcommand{\ee}{\end{equation}}

\newcommand{\grad}{\mr{grad}\,}

\newcommand{\id}{\operatorname{id}}
\newcommand{\Diff}{\operatorname{Diff}}

\newtheorem{thm}{Theorem}[section]
\newtheorem{prp}[thm]{Proposition}
\newtheorem{lem}[thm]{Lemma}
\newtheorem{cor}[thm]{Corollary}
\theoremstyle{definition}
\newtheorem{dfn}[thm]{Definition}
\theoremstyle{remark}
\newtheorem*{rmk}{Remark}
\newtheorem*{ex}{Example}

\begin{document}     
\title{Energy functionals and soliton equations for $\Gt$-forms}
\author{Hartmut Wei\ss{}} 
\address{Mathematisches Institut der LMU M\"unchen\\Theresienstra{\ss}e 39\\ D--80333 M\"unchen \\ F.R.G.}
\email{weiss@math.lmu.de}
\author{Frederik Witt} 
\address{Mathematisches Institut der Universit\"at M\"unster\\ Einsteinstra{\ss}e 62\\ D--48149 M\"unster\\ F.R.G.}
\email{frederik.witt@uni-muenster.de}

\begin{abstract}
We extend short-time existence and stability of the Dirichlet energy flow as proven in a previous paper by the authors to a broader class of energy functionals. Furthermore, we derive some monotonely decreasing quantities for the Dirichlet energy flow and investigate an equation of soliton type. In particular, we show that nearly parallel $\Gt$-structures satisfy this soliton equation and study their infinitesimal soliton deformations.
\end{abstract}

\maketitle
%
%
%
\section{Introduction}\label{introduction}
In the quest for ``special'' metrics, variational principles play an important r\^ole. A prominent example is the total scalar curvature functional on the space of Riemannian metrics, whose critical points are Ricci-flat metrics. In this article we consider various functionals defined on $\Omega^3_+(M)$, the space of {\em positive $3$-forms} on a compact, seven dimensional spin manifold $M$. These forms are sections of the fibre bundle $\Lambda^3_+T^*M \rightarrow M$ whose fibre is the {\em open} orbit $\GL(7)_+/\Gt$ of $\GL(7)_+$ acting on $\Lambda^3\R^{7*}$. Furthermore, such a section $\Omega$ induces a Riemannian metric $g_\Omega$ on $M$. We also refer to $\Omega$ as a $\Gt$-{\em structure} on $M$. The importance of this notion stems from the fact the only (irreducible) odd-dimensional instance of special holonomy comes from metrics of the form $g_\Omega$. A central problem is to find conditions which ensure the existence of a holonomy $\Gt$-metric provided necessary topological conditions are met. Such a theorem would yield an analogue of Yau's celebrated theorem~\cite{ya78} which asserts the existence of a metric with holonomy $\SU(m)$ on a K\"ahler manifold $M^{2m}$ whose first Chern class vanishes.

\medskip

The quantity we seek to extremalise is the {\em intrinsic torsion} of a positive $3$-form $\Omega$ which can be thought of as an endomorphism of $TM$ (cf.\ Section~\ref{dhfunc} for a definition). To see what this means concretely we recall that by a result of Fern\'andez and Gray~\cite{fegr82}, $\Omega$ is {\em torsion-free}, i.e.\ its intrinsic torsion vanishes, if and only if $d\Omega=0$ and $\delta_\Omega\Omega=0$ (here, $\delta_\Omega$ denotes the codifferential induced by $g_\Omega$). This, in turn, is equivalent for the holonomy of $g_\Omega$ to be contained in $\Gt$. In~\cite{wewi10} we show that the critical points of the {\em Dirichlet energy functional}
\ben
\mc D : \Omega^3_+(M) \rightarrow \R\,,\; \Omega \mapsto \frac{1}{2} \int_M (|d\Omega|_\Omega^2 + |\delta_\Omega\Omega|_\Omega^2) \vol_\Omega
\ee
(with $\vol_\Omega=\Omega \wedge \star_\Omega\Omega/7$) are precisely the torsion-free forms. Since these are absolute minimisers of $\mc D$, it is natural to consider the negative gradient flow
\begin{equation}\label{DF}
\tag{DF}
\frac{\partial}{\partial t}\, \Omega_t = -\grad \mc D(\Omega_t) =:Q(\Omega_t) 
\end{equation}
for $t \in [0,T)$, subject to some initial condition $\Omega_0 \in \Omega_+^3(M)$. Here, $-\grad$ denotes the negative $L^2$-gradient determined by $D_\Omega\mc{D}(\dot\Omega)=-\langle Q(\Omega),\dot\Omega\rangle_\Omega= - \int_M Q(\Omega) \wedge \star_\Omega \dot \Omega$ for all $\dot \Omega \in \Omega^3(M)$. The principal results of~\cite{wewi10} are these:

\begin{thm}{\bf(Short-time existence)}\label{stex}
The Dirichlet energy flow $\partial_t \Omega_t = Q(\Omega_t)$ has a unique short-time solution for any initial condition $\Omega_0 \in \Omega^3_+(M)$.
\end{thm}

In particular, for any initial condition there exists a unique solution to~\eqref{DF} on a maximal time interval $[0,T_{max})$ where  $T_{max} \in (0,\infty]$.

\begin{thm}\label{stab} {\bf(Stability)}
Let $\bar \Omega \in \Omega^3_+(M)$ be torsion-free. Then for any initial condition sufficiently close to $\bar \Omega$ in the $C^\infty$-topology the Dirichlet energy flow exists for all times and converges modulo diffeomorphisms to a torsion-free $\Gt$-structure. 
\end{thm}

\medskip

In this article we analyse the flow~\eqref{DF} further. Firstly, we derive various monotonely decreasing quantities. In particular, we show that the $W^{1,2}$-Sobolev norm $\|\Omega_t\|^2_{W^{1,2}_{\Omega_t}}$ is bounded by a monotonely decreasing bound $C_t$. Moreover, $\frac{d}{dt} C_t=0$ \iff $\Omega_t$ is torsion-free. The proof involves the functional
\ben
\mc C(\Omega)=\frac{1}{2}\int_M|\nabla^\Omega\Omega|^2_\Omega\vol_\Omega,
\ee
where $\nabla^\Omega$ is the Levi--Civita connection induced by $g_\Omega$. Its critical points are again the torsion-free positive forms, and the associated negative gradient flow has properties very similar to~\eqref{DF}. In fact, both $\mc D$ and $\mc C$ are special instances of a whole family of energy functionals. To discuss these in general, we first recall that any $\Omega\in\Omega^3_+(M)$ induces a $\Gt$-decomposition of $p$-forms $\Lambda^p=\oplus_q\Lambda^p_q$ into irreducible modules, where $q$ is the rank of the module. The corresponding module of sections will be denoted by $\Omega^p_q(M)$ (this is analogous to the decomposition into $(p,q)$-forms over an almost-complex manifold). For example, 
\be\label{pqdecomp}
\Lambda^2=\Lambda^2_7\oplus\Lambda^2_{14}\quad\mbox{and}\quad \Lambda^3=\Lambda^3_1\oplus\Lambda^3_7\oplus\Lambda^3_{27}.
\ee
Of course, $\Lambda^3_1$ is spanned by the invariant form $\Omega$. Furthermore, $\Lambda^1$ is irreducible. Since the induced Hodge-star operator $\star_\Omega$ is a $\Gt$-equivariant isomorphism $\Lambda^p\to\Lambda^{7-p}$, we immediately get the decomposition of $\Lambda^p$ for $p=4,\,5$ and $6$. In particular, we can decompose $d\Omega$ and $d\star_\Omega\Omega$ into irreducible components. Using various $\Gt$-equivariant isomorphisms we can write
\be\label{torsion}
d \Omega = \tau_0 \star_\Omega \Omega + 3 \tau_1 \wedge \Omega + \star_\Omega \tau_3
\ee
and
\be\label{cotorsion}
d \star_\Omega\Omega = 4 \tau_1 \wedge \star_\Omega\Omega + \tau_2 \wedge \Omega
\ee
(see e.g.\ Proposition 1 in~\cite{br06}) for uniquely determined {\em torsion forms} $\tau_0 \in \Omega^0_1(M)$, $\tau_1 \in \Omega^1_7(M)$, $\tau_2 \in \Omega^2_{14}(M)$ and $\tau_3 \in \Omega^3_{27}(M)$. These forms depend on $\Omega$ and can be thought of as maps from $\Omega^3_+(M)$ to $\Omega^p_q$. The $\tau_k(\Omega)$ vanish identically for all $k$ if and only if $\Omega$ is closed and coclosed, that is, if $\Omega$ is torsion-free. Note in passing that it is not obvious that $\tau_1$ appears twice in both $d\Omega$ and $d\star_\Omega\Omega$, cf.~\cite{br87}. Here, this will be a consequence of a Bianchi-type identity for $\Omega$, see the remark after Lemma~\ref{Bianchi}. We now define the energy functionals
\ben
\mathcal D_\nu:= \sum_{i=0}^3 \nu_i \mathcal D_i
\ee
with
\ben
\mathcal D_i(\Omega) := \frac{1}{2} \int_M |\tau_i|_\Omega^2 \vol_\Omega.
\ee
and $\nu=(\nu_0,\nu_1,\nu_2,\nu_3) \in \R^4$. If $\nu\in\R^4_+$, that is, all entries in $\nu$ are positive, we can prove Theorem~\ref{stex} and Theorem~\ref{stab} for the {\em generalised Dirichlet energy flow}
\begin{equation}\label{DF_nu}
\tag{$\text{DF}_\nu$}
\frac{\partial}{\partial t}\, \Omega_t = Q_\nu(\Omega_t),
\end{equation}
see Theorem~\ref{nuste} and Theorem~\ref{nustab}. The flow~\eqref{DF} is just the special case for $\nu=(7,84,1,1)$. However, we shall write $\mc D$ and $Q$ for $\mc D_{\nu}$ and $Q_{\nu}$ in this case to be consistent with~\cite{wewi10}.

\medskip

To obtain concrete solutions to \eqref{DF_nu}, we consider the equation
\ben
Q_\nu(\Omega_0) =\mu_0\Omega_0 + \mathcal L_{X_0} \Omega_0
\ee
for some real constant $\mu_0$ and vector field $X_0$ (with $\mc{L}_{X_0}$ the Lie derivative along $X_0$). In analogy with Ricci-flow we call this the {\em $\mc D_\nu$-soliton equation}. A $\mc D$-soliton, where $\mc D$ is the original Dirichlet energy functional, will be simply called a {\em$\Gt$-soliton}. For a $\mc D_\nu$-soliton $\Omega_0$ as initial condition, the solution to \eqref{DF_nu} has the form $\Omega_t=\mu(t)\Omega_0$ with $\mu(t) \searrow 0$ as $t \nearrow T_{max}$, and so becomes singular. As in the Ricci-flow case, one expects $\Gt$-solitons to play a major r\^ole in the study of finite time singularities. We first show that any $\Gt$-soliton is necessarily of the form $Q(\Omega)=\mu\Omega$. This is precisely the condition to be a critical point for $\mc D$ subject to the constraint that the total volume $\int_M\vol_\Omega$ equals $1$. Furthermore, any such $\Gt$-soliton is either steady, i.e.\ $\mu_0=0$, in which case the flow is constant and thus exists trivially for all times, or shrinking, i.e.\ $\mu_0<0$. In this case the flow collapses in finite time. Our main result is that {\em nearly parallel $\Gt$-structures} (i.e.\ $\Gt$-structures for which all torsion forms but $\tau_0$ vanish) are $\Gt$-solitons in the sense above (cf.\ Theorem~\ref{weakg2sol}). For example, the $7$-sphere with the round metric is nearly parallel. In general, nearly parallel $\Gt$-structures induce Einstein metrics with positive Einstein constant. However, we do not know whether a soliton is necessarily of this type. Finally, we investigate the premoduli space of $\Gt$-soliton deformations at a nearly parallel $\Gt$-structure. As in the Einstein case we can prove that the premoduli space is a real-analytic subset of some finite dimensional real analytic submanifold (cf.\ Theorem~\ref{premodstructure}). Any infinitesimal Einstein deformation of a nearly parallel $\Gt$-structure gives an infinitesimal soliton deformation, but again we do not know whether the converse holds.

\medskip

{\em Conventions.} 
(i) In this paper we shall only encounter irreducible $\Gt$-representation spaces of dimension equal or less than $27$. In this range, an irreducible $\Gt$-representation is uniquely determined by its dimension $q$. For instance, the space of symmetric $2$-tensors $\odot^2\R^{7*}$ can be decomposed into the line spanned by the identity and the $27$-dimensional irreducible space of tracefree $2$-tensors $\odot^2_0\R^{7*}$, which is thus isomorphic to $\Lambda^3_{27}\R^{7*}$. Consequently, the module of endomorphisms can be decomposed into
\be\label{endodecomp}
\R^{7*}\otimes\R^{7*}=\odot^2\R^{7*}\oplus\Lambda^2\R^{7*}=\Lambda^3_0\oplus\Lambda^3_{27}\oplus\Lambda^3_7\oplus\Lambda^2_{14}.
\ee
We denote projection onto irreducible components by $[\,\cdot\,]_q$. For example, a $3$-form $\dot\Omega\in\Omega^3(M)$ can be decomposed into $\dot\Omega=[\dot\Omega]_1\oplus [\dot\Omega]_7\oplus[\dot\Omega]_{27}$ and an endomorphism $\dot A$ into $[\dot A]_1\oplus[\dot A]_7\oplus[\dot A]_{14}\oplus[\dot A]_{27}$.

(ii) If $F:\Omega^3_+(M)\to E$ is a smooth map between Fr\'echet spaces we often write $\dot F_\Omega$ for $D_\Omega F(\dot\Omega)$, the linearisation of $F$ at $\Omega$ evaluated in $\dot\Omega\in\Omega^3(M)$. For example, for the map $\Theta:\Omega^3_+(M)\to\Omega^4(M)$ which sends $\Omega$ to $\Theta(\Omega)=\star_\Omega\Omega$, we get 
\be\label{thetader}
\dot \Theta_\Omega = \star_\Omega p_\Omega(\dot \Omega)
\ee
with
\ben
p_\Omega(\dot \Omega) = \frac{4}{3} [\dot \Omega]_1 + [\dot \Omega]_7 - [\dot \Omega]_{27}.
\ee

Another example is $Q:\Omega^3_+(M)\to\Omega^3(M)$, the negative gradient of $\mc D$, given by
\be\label{qop}
Q(\Omega)=-\delta_\Omega d\Omega-p_\Omega(d\delta_\Omega\Omega)-q_\Omega(\nabla^\Omega\Omega),
\ee
where $q_\Omega$ is determined by the identities 
\be\label{quadratic}
\langle \dot \Omega, q_\Omega(\nabla^\Omega\Omega)\rangle_\Omega = \frac{1}{2} \bigl( \langle \dot\star_\Omega d\Omega, \star_\Omega d\Omega \rangle_\Omega + \langle \dot\star_\Omega d\star_\Omega\Omega, \star_\Omega d\star_\Omega\Omega \rangle_\Omega \bigr)
\ee
to hold for all $\dot \Omega \in \Omega^3(M)$. 
%
%
%
\section{The Dirichlet energy and the Hitchin functional}\label{dhfunc}
%
\subsection{The torsion forms of a positive 3-form}
Recall that $\nabla^\Omega\Omega$ is a section of $\Lambda^1\otimes\Lambda^3_7$ and hence may be written as $\nabla^\Omega\Omega=T(\Omega)$ for a uniquely determined tensor field $T\in\Gamma(\Lambda^1\otimes\Lambda^2_7)$, the {\em intrinsic torsion} of the $\Gt$-structure (cf.\ for example~\cite{br06}). Here the $\Lambda^2_7$ factor of $T$ acts, seen as an element in $\Lambda^2\cong\mf{so}(7)$, the Lie algebra of $\SO(7)$, equivariantly in the standard way on $\Omega$ and gives an element in $\Lambda^3_7$. The module $\Lambda^1_7 \otimes \Lambda^3_7$ decomposes as $\Lambda^0_1 \oplus \Lambda^1_7 \oplus \Lambda^2_{14} \oplus \Lambda^3_{27}$ into $\Gt$-irreducible ones. Hence 
\ben
\nabla^\Omega\Omega = \xi_1 + \xi_7 + \xi_{14} + \xi_{27},
\ee
where $\xi_i$ denotes the projection of $\xi:=\nabla^\Omega\Omega$ onto the corresponding irreducible summand. The $\xi_k$ are thus the irreducible components of the intrinsic torsion $T$ under the embedding $T\mapsto T(\Omega)$.

\begin{prp}\label{d_delta_norm}
Let $\Omega \in \Omega^3_+(M)$ be a positive 3-form. Then the following holds:

{\rm(i)} One has
\ben
|d\Omega|_\Omega^2 = 7 \tau_0^2 + 36 |\tau_1|_\Omega^2 + |\tau_3|_\Omega^2
\ee
and
\ben
|\delta_\Omega\Omega|_\Omega^2 = 48 |\tau_1|_\Omega^2 + |\tau_2|_\Omega^2.
\ee
In particular,
\be\label{D1_density}
|d\Omega|_\Omega^2 + |\delta_\Omega\Omega|_\Omega^2 = 7 \tau_0^2 + 84 |\tau_1|_\Omega^2 + |\tau_2|_\Omega^2 + |\tau_3|_\Omega^2. 
\ee

{\rm(ii)} One has
\be\label{D2_density}
|\nabla^\Omega\Omega|_\Omega^2 = \frac{7}{4} \tau_0^2 + 24 |\tau_1|_\Omega^2 + 2 |\tau_2|_\Omega^2 + 2 |\tau_3|_\Omega^2.   
\ee
\end{prp}
\begin{proof}
(i) Clearly
\ben
|d\Omega|_\Omega^2 = \tau_0^2 |\star_\Omega\!\Omega|_\Omega^2 + 9|\tau_1 \wedge \Omega|_\Omega^2 + |\tau_3|_\Omega^2,
\ee
which using $|\!\star_\Omega\!\Omega|_\Omega^2 =|\Omega|_\Omega^2=7$ and $|\tau_1\wedge\Omega|_\Omega^2 = 4 |\tau_1|_\Omega^2$ (cf.\ for instance equation (15) in \cite{wewi10}) yields the first equation. Similarly,
\ben
|\delta_\Omega\Omega|_\Omega^2 = |d \star_\Omega \Omega|_\Omega^2 = 16 | \tau_1 \wedge \star_\Omega \Omega|_\Omega^2 + |\tau_2 \wedge \Omega|_\Omega^2
\ee
as $|\tau_1 \wedge \star_\Omega \Omega|_\Omega^2 = 3 |\tau_1|_\Omega^2$ (cf.\ equation (15) in \cite{wewi10}) and $|\tau_2 \wedge \Omega|_\Omega^2=|\tau_2|_\Omega^2$, for $\Lambda^2_{14} = \{ \alpha \in \Lambda^2\,|\,\alpha\wedge\Omega=-\star_\Omega\alpha\}$.

\medskip

(ii) Let $\varepsilon: \Lambda^1 \otimes \Lambda^k \rightarrow \Lambda^{k+1}$ and $\iota: \Lambda^1 \otimes \Lambda^k \rightarrow \Lambda^{k-1}$ denote exterior resp.~interior multiplication. Then $d\Omega = \epsilon(\xi)$ and $\delta_\Omega\Omega = - \iota (\xi)$. Since $\varepsilon$ and $\iota$ are $\GL$-equivariant, one has more precisely
\ben
d \Omega = \ep(\xi_1) + \ep(\xi_7) + \ep(\xi_{27})
\ee
and
\ben
\delta_\Omega \Omega = -\iota(\xi_7) - \iota (\xi_{14}).
\ee
We need to calculate the length distortion of the maps $\xi$ and $\iota$ on the irreducible summands. We claim that
\ben
|\ep(\xi_1)|_\Omega^2 = 4|\xi_1|_\Omega^2\,,\;|\ep(\xi_7)|_\Omega^2 =\frac{3}{2} |\xi_7|_\Omega^2\,,\;|\ep(\xi_{27})|_\Omega^2 = \frac{1}{2}|\xi_{27}|_\Omega^2  
\ee
and
\ben
|\iota(\xi_7)|_\Omega^2 = 2 |\xi_7|_\Omega^2\,,\;|\iota(\xi_{14})|_\Omega^2 = \frac{1}{2} |\xi_{14}|_\Omega^2. 
\ee
To establish these we consider the map $f:\Lambda^1\otimes\Lambda^1\to\Lambda^1\otimes\Lambda^3_7$ which to $v\otimes w$ assigns $v\otimes(w\llcorner\star_\Omega\Omega)$. The module of symmetric endomorphisms $\odot^2$ which is spanned by $v\otimes w+w\otimes v$ can be decomposed into the tracefree endomorphisms $\odot^2_0$ and multiples of the identity. A ($\GL(7)$-)equivariant projection $\pi_0:\odot^2\to\odot^2_0$ is given by $\pi_0(a)=a-\mathrm{Tr}(a)\Id/7$. We want to compute $|\epsilon(f(\pi_0(a)))|^2$ and $|f(\pi_0(a))|^2$ for $a\in\odot^2$. It suffices to do this for elements of the form $e_i\otimes e_j+e_j\otimes e_i$ for some orthonormal basis $e_1,\ldots,e_7$ of $\Lambda^1$. Furthermore, since $\Gt$ acts transitively on pairs of orthonormal vectors, we need to consider the element $e_1\otimes e_2+e_2\otimes e_1$ only, which is already in $\odot^2_0$. Thus
\ben
|f(e_1\otimes e_2+e_2\otimes e_1)|^2=|e_1\otimes(e_2\llcorner\star_\Omega\Omega)+e_2\otimes(e_1\llcorner\star_\Omega\Omega)|^2=8
\ee
while
\ben
|\epsilon(f(e_1\otimes e_2+e_2\otimes e_1))|^2=|e_1\wedge(e_2\llcorner\star_\Omega\Omega)+e_2\wedge(e_1\llcorner\star_\Omega\Omega)|^2=4,
\ee
whence the distortion factor $1/2$ as claimed above. In the same vein, consider the projection $\pi^2_{14}:\Lambda^2\to\Lambda^2_{14}$ given by $\pi^2_{14}(\alpha)=(2\alpha-\star_\Omega(\alpha\wedge\Omega))/3$. Then
\ben
|f(e_1\otimes e_2-e_2\otimes e_1)|^2=|e_1\otimes(e_2\llcorner\star_\Omega\Omega)-e_2\otimes(e_1\llcorner\star_\Omega\Omega)|^2=8
\ee
and
\ben
|\iota(f(e_1\otimes e_2-e_2\otimes e_1))|^2=|e_1\llcorner(e_2\llcorner\star_\Omega\Omega)-e_2\llcorner(e_1\llcorner\star_\Omega\Omega)|^2=4,
\ee
giving again the distortion factor $1/2$. Either by proceeding as before or by using the transitivity of $\Gt$ on the sphere of its vector representation we deduce the remaining coefficients. Therefore
\ben
|d\Omega|_\Omega^2 = 4|\xi_1|_\Omega^2 + \frac{3}{2}|\xi_7|_\Omega^2 + \frac{1}{2} |\xi_{27}|_\Omega^2
\ee
and 
\ben
|\delta_\Omega\Omega|_\Omega^2 = 2|\xi_7|_\Omega^2 + \frac{1}{2}|\xi_{14}|_\Omega^2.
\ee
Comparing this with the formul{\ae}~\eqref{torsion} and~\eqref{cotorsion} we get:
\ben
|\xi_1|_\Omega^2 = \frac{7}{4} \tau_0^2\,,\;|\xi_7|_\Omega^2 = 24 |\tau_1|_\Omega^2\,,\; |\xi_{14}|_\Omega^2 = 2 |\tau_2|_\Omega^2\,,\; |\xi_{27}|_\Omega^2 = 2|\tau_3|_\Omega^2.
\ee
Since clearly 
\ben
|\nabla^\Omega\Omega|_\Omega^2 = |\xi_1|_\Omega^2 + |\xi_7|_\Omega^2 + |\xi_{14}|_\Omega^2 +|\xi_{27}|_\Omega^2
\ee
the result follows.
\end{proof}

\begin{rmk}
The previous proposition provides an alternative proof of the result of Fern\'andez and Gray mentioned in the introduction: For $\Omega \in \Omega^3_+(M)$ one has $\nabla^\Omega\Omega=0$ if and only if $d\Omega=\delta_\Omega\Omega=0$, since both equations are equivalent to $\tau_0=\tau_1=\tau_2=\tau_3=0$. By standard holonomy theory, $\nabla^\Omega\Omega=0$ is equivalent to $g_\Omega$ having holonomy contained in $\Gt$.
\end{rmk}
%
\subsection{Monotone quantities}
For any smooth family $\Omega_t$ we can write
\ben
\partial_t\Omega_t=3f_t\Omega_t+\star_{\Omega_t}(\alpha_t\wedge\Omega_t)+\gamma_t
\ee
for uniquely determined quantities $f_t\in C^\infty(M)$, $\alpha_t\in\Omega^1(M)$ and $\gamma_t\in\Omega^3_{27,\Omega_t}(M)$ depending smoothly on $t$. These are called the {\em deformation forms} of $\Omega_t$. In particular, the evolution of the associated volume form is given by
\ben
\partial_t\vol_{\Omega_t}=7f_t\vol_{\Omega_t},
\ee
see e.g.~\cite{br06}. For a solution $\Omega_t$ to~\eqref{DF} we have
\ben
g_{\Omega_t}(Q(\Omega_t),\Omega_t) = 3 f_t g_{\Omega_t}(\Omega_t,\Omega_t) = 21 f_t
\ee
and hence
\be\label{evol_vol}
\partial_t \vol_{\Omega_t}=\tfrac{1}{3} g_{\Omega_t}(Q(\Omega_t), \Omega_t) \vol_{\Omega_t}.
\ee
Alternatively, use that the differential of the map $\phi:\Lambda^3_+\to\Lambda^7$ sending $\Omega$ to $\vol_\Omega$ is given by
\be\label{phider}
D_\Omega\phi(\dot\Omega)=\tfrac{1}{3}\dot\Omega\wedge\star_\Omega\Omega,
\ee
cf.~\cite{hi01}. The {\em Hitchin functional} is defined by
\ben
\mathcal H : \Omega^3_+(M) \rightarrow \R\,,\; \Omega \mapsto \int_M \vol_\Omega,
\ee
i.e.\ it associates with $\Omega\in\Omega^3_+(M)$ its total volume. We find that the value of the Hitchin functional is monotone and convex along a solution to the Dirichlet energy flow:

\begin{prp}\label{H_monotone}
If $(\Omega_t)_{t \in [0,T)}$ is a solution to~\eqref{DF}, then
\ben
\frac{d}{dt} \mathcal H(\Omega_t) \leq 0\mbox{ and }\frac{d^2}{dt^2} \mathcal H(\Omega_t) \geq 0
\ee
for all $t \in [0,T)$. Further, $\frac{d}{dt}\bigr\vert_{t=t_0}\mathcal H(\Omega_t)= 0$ if and only if $\Omega_{t_0}$ is torsion-free.
\end{prp}
\begin{proof}
Using equation \eqref{evol_vol} we get
\begin{align*}
\frac{d}{dt} \mathcal H (\Omega_t) & = \int_M \frac{\partial}{\partial t} \vol_{\Omega_t}\\
& = \frac{1}{3} \int_M g_{\Omega_t}(Q(\Omega_t), \Omega_t)\vol_{\Omega_t}\\
& = - \tfrac{1}{3} D_{\Omega_t} \mathcal D(\Omega_t)
\end{align*}
Since $\mathcal D$ is positively homogeneous, i.e.~$\mathcal D(\lambda\Omega) = \lambda^{5/3} \mathcal D(\Omega)$ for $\lambda >0$, one has $D_\Omega\mathcal D(\Omega) = \tfrac{5}{3} \mathcal D(\Omega)$ by Euler's formula, cf.~the proof of Corollary 4.3 in~\cite{wewi10}. Hence
\begin{equation}\label{der_Hitchin}
\frac{d}{dt} \mathcal H (\Omega_t) = -\tfrac{5}{9} \mathcal D(\Omega_t) \leq 0
\end{equation}
with equality if and only if $\Omega_t$ is torsion-free. Furthermore,
\ben
\frac{d^2}{dt^2} \mathcal H (\Omega_t) = -\tfrac{5}{9}D_{\Omega_t}\mathcal D(Q(\Omega_t))=\tfrac{5}{9}\|Q(\Omega_t)\|^2_{\Omega_t}
\ee
which is always non-negative.
\end{proof}

Equation \eqref{der_Hitchin} has the following noteworthy consequence for a long-time solution to the Dirichlet energy flow: Suppose that $\Omega_t$ is a solution to \eqref{DF} on $[0,\infty)$. Then, since $\mc{D}(\Omega_t)$ is monotonely decreasing, the limit
\ben
\mc{D}_{\infty}:=\lim_{t \rightarrow \infty} \mc{D}(\Omega_t)\geq 0
\ee
exists. In fact, we have
\begin{cor}\label{D_infty}
If $(\Omega_t)_{t \in [0,\infty)}$ is a solution to \eqref{DF}, then $\mc{D}_\infty=0$.
\end{cor}

\begin{proof}
Assume to the contrary that $\mc{D}_\infty>0$. Then $\mc{D}(\Omega_t) \geq \mc{D}_\infty>0$ for all $t \in [0,\infty)$. Hence, by equation \eqref{der_Hitchin}, $\frac{d}{dt} \mc{H}(\Omega_t) \leq - \frac{5}{9} \mc{D}_\infty <0$ for all $t$, and therefore
\ben
\mc{H}(\Omega_t) \leq \mc{H}(\Omega_0) - \frac{5}{9} \mc{D}_\infty t.
\ee
In particular, $\mc{H}(\Omega_t)$ becomes negative in finite time. Contradiction!
\end{proof}

\begin{rmk}
As an example communicated to us by Joel Fine shows, long-time existence is not sufficient to imply convergence to a critical point, cf.~\cite{fine}. It is obtained by restricting the Dirichlet energy functional $\mc D$ to the space of $\SO(4)$-invariant forms on $\R^4\times SO(3)$. Using Lemma~\ref{invariance}, the flow equations can be reduced to a system of nonlinear ODEs which can be explicitly solved and whose solutions project down to $T^4\times SO(3)$. This is related to the failure of the Dirichlet energy functional to satisfy the Palais--Smale condition. If, however, $\lim_{t \rightarrow \infty} \Omega_t = \Omega_\infty \in \Omega^3_+(M)$, say w.r.t.~the $C^1$-topology, then Corollary \ref{D_infty} suffices to conclude that $\Omega_\infty$ is torsion-free.
\end{rmk}

As for the Dirichlet energy functional, we may set
\ben
\mc{H}_\infty:=\lim_{t \rightarrow \infty} \mc{H}(\Omega_t) \geq 0
\ee
for a solution $\Omega_t$ to \eqref{DF} on $[0,\infty)$. Here two cases may occur:
\begin{enumerate}
\item $\mc{H}_\infty >0$
\item $\mc{H}_\infty = 0$
\end{enumerate}
A prototypical example for the first case is a solution converging to a torsion-free $\Gt$-structure as $t \rightarrow \infty$. Such solutions exist as a consequence of Theorem \ref{stab}, our stability result for the Dirichlet energy flow. A solution fitting into the second case is provided by Fine's example, cf.~\cite{fine}.

\medskip

A further consequence of equation \eqref{der_Hitchin} is that the value of the Hitchin functional decays at most linearly along a solution to the Dirichlet energy flow:

\begin{cor}
If $(\Omega_t)_{t \in [0,T)}$ is a solution to \eqref{DF}, then
\ben
\mc{H}(\Omega_0)\geq\mathcal H(\Omega_t) \geq \mathcal H(\Omega_0) - \tfrac{5}{9} \mathcal D(\Omega_0)t
\ee
for all $t \in [0,T)$. In particular, for all $\varepsilon>0$ there exists $\delta >0$ such that $\mathcal H(\Omega_t) \geq \delta $ for all $t \in [0, t_0 - \varepsilon]$ with $t_0 = \min \{T, \frac{9}{5} \frac{\mathcal H(\Omega_0)}{\mathcal D(\Omega_0)}\}$.
\end{cor}

\begin{proof}
Since the Dirichlet energy flow is the negative gradient flow of $\mathcal D$, one clearly has 
\ben
\frac{d}{dt} \mathcal D(\Omega_t) \leq 0
\ee
for all $t \in [0,T)$, in particular $\mathcal D(\Omega_t) \leq \mathcal D(\Omega_0)$. Hence by equation \eqref{der_Hitchin}
\ben
\frac{d}{dt} \mathcal H(\Omega_t) \geq -\frac{5}{9} \mathcal D(\Omega_0),
\ee 
and the claim follows by integration.
\end{proof}

\begin{rmk}
If one knew exponential decay of $\mathcal D (\Omega_t)$ for a solution to \eqref{DF} on $[0,\infty)$ beforehand, then $\mathcal H(\Omega_t)$ would be bounded from below: Assuming $\mathcal D(\Omega_t) \leq C e^{-\lambda t}$ for constants $C,\lambda>0$ and using equation \eqref{der_Hitchin} once again, one gets 
\begin{align*}
\mathcal H(\Omega_t) &\geq \mathcal H(\Omega_0) - \frac{5}{9} \int_0^t Ce^{-\lambda\tau} d\tau\\
&= \mathcal H(\Omega_0) - \frac{5}{9} \frac{C}{\lambda} (1 - e^{-\lambda t})\\
& \geq \mathcal H(\Omega_0) - \frac{5}{9} \frac{C}{\lambda}
\end{align*}
for all $t \in [0, \infty)$. This would be particularly useful if one could choose $C$ and $\lambda$ in such a way that $\delta:=\mathcal H(\Omega_0) - \frac{5}{9} \frac{C}{\lambda}>0$.
\end{rmk}

In~\cite{br06} it is shown that the scalar curvature of the metric $g_\Omega$ is given by
\ben
s_{g_\Omega} = 12 \delta_\Omega \tau_1 + \frac{21}{8} \tau_0^2 + 30 |\tau_1|_\Omega^2 - \frac{1}{2}|\tau_2|_\Omega^2 -\frac{1}{2} |\tau_3|_\Omega^2
\ee
(cf.~(4.28) loc.~cit.). Thus, by Stokes theorem the total scalar curvature 
\ben
\mc{S}(\Omega):=\int_Ms_{g_\Omega}\vol_\Omega
\ee
of $g_\Omega$
is given by
\be\label{scalcurv}
\mc{S}(\Omega)=\int_M \Bigl(\frac{21}{8} \tau_0^2 + 30 |\tau_1|_\Omega^2 - \frac{1}{2}|\tau_2|_\Omega^2 -\frac{1}{2} |\tau_3|_\Omega^2\Bigr) \vol_\Omega.
\ee
On the other hand, by Proposition~\ref{d_delta_norm}, we have
\ben
\mathcal{D}(\Omega)= \int_M \Bigl( \frac{7}{2} \tau_0^2  + 42 |\tau_1|_\Omega^2 + \frac{1}{2}|\tau_2|_\Omega^2 + \frac{1}{2} |\tau_3|_\Omega^2\Bigr) \vol_\Omega.
\ee
Comparing coefficients immediately yields 
\begin{lem}
Let $\Omega \in \Omega^3_+(M)$ be a positive 3-form. Then $|\mc{S}(\Omega)|\leq\mc{D}(\Omega)$.
\end{lem}
Using the monotonicity of $\mc{D}$ and Corollary \ref{D_infty} we obtain
\begin{cor}
The absolute value of the total scalar curvature $\mc{S}(\Omega_t)$ is bounded by a monotonely decreasing quantity along a solution $(\Omega_t)_{t \in [0,T)}$ to~\eqref{DF}. If $\Omega_t$ is defined on $[0,\infty)$, then $\lim_{t \rightarrow \infty} \mc{S}(\Omega_t)=0$.
\end{cor}
If we define 
\ben
\mc{C}(\Omega):=\frac{1}{2}\int_M | \nabla^\Omega \Omega|_\Omega^2 \vol_\Omega
\ee 
we get from equation~\eqref{D1_density}
and from equation~\eqref{D2_density}
\begin{align*}
\mathcal{C}(\Omega) &= \int_M \Bigl( \frac{7}{8} \tau_0^2  + 12 |\tau_1|_\Omega^2 + |\tau_2|_\Omega^2 + |\tau_3|_\Omega^2\Bigr) \vol_\Omega\\
&= \mathcal{D}(\Omega) + \int_M \Bigl(-\frac{21}{8} \tau_0^2 -30 |\tau_1|_\Omega^2 + \frac{1}{2}|\tau_2|_\Omega^2 + \frac{1}{2} |\tau_3|_\Omega^2\Bigr) \vol_\Omega\\
&= \mathcal{D}(\Omega) - \mathcal{S}(\Omega). 
\end{align*}
Furthermore, we remark that
\ben
2\mc{C}(\Omega)+7\mc{H}(\Omega)=\|\Omega\|^2_{W^{1,2}_\Omega},
\ee
whence
\ben
0\leq\|\Omega\|^2_{W^{1,2}_\Omega}\leq4\mc{D}(\Omega)+7\mc{H}(\Omega)\leq8\|\Omega\|^2_{W^{1,2}_\Omega}.
\ee
In particular, we find that along a solution to the Dirichlet energy flow:

\begin{prp}
Let $(\Omega_t)_{t \in [0,T)}$ be a solution to~\eqref{DF}. Then
\ben
\|\Omega_t\|^2_{W^{1,2}_{\Omega_t}}\leq C_t\leq C_0
\ee
for the monotonely decreasing bound $C_t:=4\mc{D}(\Omega_t)+7\mc{H}(\Omega_t)$. Furthermore, one has $\frac{d}{dt}\bigr\vert_{t=t_0} C_t=0$ \iff $\Omega_{t_0}$ is torsion-free.
\end{prp}
\begin{proof}
The first assertion follows directly from the discussion above. Secondly, $\frac{d}{dt} C_t=4\frac{d}{dt}\mc{D}(\Omega_t)+7\frac{d}{dt}\mc{H}(\Omega_t)\leq0$ with equality \iff $\frac{d}{dt}\mc{D}(\Omega_t)=0$ and $\frac{d}{dt}\mc{H}(\Omega_t)=0$, whence the result by Proposition~\ref{H_monotone}.
\end{proof}
%
\subsection{The generalised Dirichlet energy flow}
The energy functionals $\mc{D}$ and $\mc{C}$ considered above are special instances of the functional
\ben
\mathcal D_{\lambda}:= \sum_{i=0}^3 \nu_i \mathcal D_i
\ee
with
\ben
\mathcal D_i(\Omega) := \frac{1}{2} \int_M |\tau_i|_\Omega^2 \vol_\Omega.
\ee
and $\nu=(\nu_0,\nu_1,\nu_2,\nu_3) \in \R^4$. More specifically, one has
\ben
\mathcal D = 7 \mathcal D_0 + 84 \mathcal D_1 + \mathcal D_2 + \mathcal D_3
\ee
and
\ben
\mathcal C=\frac{7}{4} \mathcal D_0 + 24 \mathcal D_1 + 2\mathcal D_2 + 2\mathcal D_3.
\ee
We call the functional $\mathcal{D}_\nu$ the {\em generalised Dirichlet energy functional} associated with the parameter $\nu \in \R^4$. The aim of this section is to further analyse this family of functionals. In particular, we prove generalised versions of Theorem~\ref{stex} and Theorem~\ref{stab} for $\mc{D}_\nu$ for $\nu\in\R^4_+$.

\medskip

Set $Q_i(\Omega):= - \grad \mathcal D_i(\Omega)$, $i=0,1,2,3$ and $Q_\nu(\Omega):= - \grad \mathcal D_\nu(\Omega)$ for $\nu \in \R^4$. The functional $\mathcal{D}_\nu$ shares the same basic properties with $\mathcal{D}$: It is $\Diff(M)_+$-invariant and positively homogeneous, i.e.~$\mathcal{D}_\nu(\mu \Omega) = \mu^{\frac{5}{3}} \mathcal{D}_\nu(\Omega)$ for $\mu\in \R_+$. 

\medskip

Next we consider the negative gradient flow of the generalised Dirichlet energy functional
\begin{equation}\label{gDF}
\tag{$\text{DF}_\nu$}
\frac{\partial}{\partial t}\, \Omega_t = Q_\nu(\Omega_t) 
\end{equation}
for $\nu \in \R^4$, subject to some initial condition $\Omega_0 \in \Omega^3_+(M)$. We call the flow equation \eqref{gDF} the {\em generalised Dirichlet energy flow}. 

\medskip

For $\nu \in \R^4_+$ the generalised Dirichlet energy flow behaves much like the ordinary Dirichlet energy flow. In this case Euler's formula implies as for $Q$ (corresponding to $\mc{D}$) that $Q_\nu(\Omega)=0$ holds if and only if $\Omega$ is torsion-free. As a first result we have:

\begin{lem}
The flow equation \eqref{gDF} is weakly parabolic for $\nu \in \R^4_{\geq 0}$, i.e.
\ben
-g_\Omega( \sigma(D_\Omega Q_\nu)(x,\xi)\dot \Omega,\dot \Omega) \geq 0
\ee
for all $x \in M$, $\xi \in T^*_xM$ and $\dot \Omega \in \Lambda^3T^*_xM$.
\end{lem}

\begin{proof}
According to Proposition \ref{d_delta_norm} one has
\ben
|[d\Omega]_1|_\Omega^2 = 7 \tau_0^2,\, |[d\Omega]_7|_\Omega^2 = 36|\tau_1|_\Omega^2,\, |[d\Omega]_{27}|_\Omega^2 = |\tau_3|_\Omega^2
\ee
and
\ben
|[\delta_\Omega\Omega]_{7}|^2 = 48|\tau_1|_\Omega^2,\,|[\delta_\Omega\Omega]_{14}|^2 = |\tau_2|_\Omega^2.
\ee 
Therefore
\begin{align*}
7 \cdot \mathcal D_0(\Omega) &= \frac{1}{2} \int_M |[d\Omega]_1|_\Omega^2 \vol_\Omega,\\
36 \cdot \mathcal D_1(\Omega) &= \frac{1}{2} \int_M |[d\Omega]_7|_\Omega^2 \vol_\Omega,\\
\mathcal D_3(\Omega) &= \frac{1}{2} \int_M |[d\Omega]_{27}|_\Omega^2 \vol_\Omega
\end{align*}
and
\ben
48 \cdot \mathcal D_1(\Omega) = \frac{1}{2} \int_M |[\delta_\Omega\Omega]_7|_\Omega^2 \vol_\Omega,\,
\mathcal D_2(\Omega) = \frac{1}{2} \int_M |[\delta_\Omega\Omega]_{14}|_\Omega^2 \vol_\Omega.
\ee
Linearising as in \cite {wewi10} we get
\begin{alignat*}{2}
-\sigma (D_\Omega Q_0) (x,\xi) \dot \Omega &=  \frac{1}{7}\xi \llcorner [\xi \wedge \dot \Omega ]_1,\quad &-\sigma (D_\Omega Q_1) (x,\xi) \dot \Omega &=  \frac{1}{36}\xi \llcorner [\xi \wedge \dot \Omega ]_7,\\
-\sigma (D_\Omega Q_2) (x,\xi) \dot \Omega &=  p_\Omega(\xi \wedge [\xi \llcorner p_\Omega \dot \Omega ]_{14}),\quad & -\sigma (D_\Omega Q_3) (x,\xi) \dot \Omega &=  \xi \llcorner [\xi \wedge \dot \Omega ]_{27}.
\end{alignat*}
Now for $k=1,7,27$ we have for $\xi \in T_x^*M$
\ben
g_\Omega(\xi \llcorner [ \xi \wedge \dot \Omega]_k, \dot \Omega ) = |[\xi \wedge \dot \Omega]_k|_\Omega^2 \geq 0
\ee
and for $k=14$
\ben
g_\Omega( p_\Omega(\xi \wedge [\xi \llcorner p_\Omega \dot \Omega ]_{14}), \dot \Omega) = |[\xi \llcorner p_\Omega \dot \Omega ]_{14}|_\Omega^2 \geq 0.
\ee
Since $D_\Omega Q_\nu = \sum_{i=0}^3 \nu_i D_\Omega Q_i$, the result follows.
\end{proof}

Breaking the diffeomorphism invariance one gets: 

\begin{thm}\label{nuste}
The generalised Dirichlet energy flow $\partial_t \Omega_t = Q_\nu(\Omega_t)$ has a unique short-time solution for $\nu\in\R^4_+$ and any initial condition $\Omega_0 \in \Omega^3_+(M)$.
\end{thm}

\begin{proof}
We employ DeTurck's trick as in \cite{wewi10}. Given some background $\Gt$-structure $\bar\Omega \in
\Omega^3_+(M)$ (e.g.~the initial condition $\Omega_0$) we consider the vector field
\ben
X(\Omega)=-(\delta_{\bar\Omega}\Omega) \llcorner \bar\Omega.
\ee
For $\varepsilon(\nu)=\min_{i=0,1,2,3}\nu_i/36$ we set $\Lambda(\Omega):=\mathcal L_{X(\Omega)}\Omega$ and
\ben
\widetilde Q_{\nu}(\Omega) := Q_\nu(\Omega) + \varepsilon(\nu) \Lambda(\Omega).
\ee
Then $D_\Omega\widetilde Q_{\nu} = D_\Omega Q_\nu + \varepsilon(\nu) D_\Omega \Lambda$. For $\xi \in T_x^*M$ with $|\xi|_\Omega=1$ we find that
\begin{align*}
-g_\Omega(\sigma(D_\Omega Q_\nu)(x,\xi)\dot \Omega, \dot \Omega) &= - \sum_{i=0}^3 \nu_i g_\Omega(\sigma(D_\Omega Q_i)(x,\xi)\dot \Omega, \dot \Omega)\\
 & \geq - \varepsilon(\nu) g_\Omega(\sigma(D_\Omega Q)(x,\xi)\dot \Omega, \dot \Omega)
\end{align*}
and hence
\begin{align*}
&-g_\Omega(\sigma(D_\Omega \widetilde Q_{\nu})(x,\xi)\dot \Omega, \dot \Omega)\\
\geq &-\varepsilon(\nu) g_\Omega(\sigma(D_\Omega Q)(x,\xi)\dot \Omega, \dot \Omega)  - \varepsilon(\nu) g_\Omega(\sigma (D_\Omega\Lambda)(x,\xi)\dot \Omega, \dot \Omega)\\ 
= &-\varepsilon(\nu) g_\Omega(\sigma(D_\Omega \widetilde Q)(x,\xi)\dot \Omega, \dot \Omega) \geq \varepsilon(\nu) |\dot \Omega|_\Omega^2,
\end{align*}
where the last line follows from Lemma 5.7 in \cite{wewi10}. 

\medskip

This shows that
the flow equation $\partial_t \widetilde\Omega_t = \widetilde Q_{\nu}(\widetilde\Omega_t)$
is strongly parabolic. Standard methods, see for instance \cite{top}, now yield a unique short-time
solution $\widetilde \Omega_t$. A short-time solution $\Omega_t$ for the original flow
equation $\partial_t \Omega_t = Q_\nu(\Omega_t)$ is then obtained by integrating the time-dependent vector field $X(\widetilde\Omega_t)$ and pulling back $\widetilde\Omega_t$ by the
corresponding family of diffeomorphisms, cf.~\cite{wewi10} for details. 

\medskip

The proof of uniqueness given in \cite{wewi10} for the Dirichlet energy flow applies
without change to yield uniqueness of the solution $\Omega_t$ on short time intervals.
\end{proof}

Finally, as in \cite{wewi10} we also get a stability result:

\begin{thm}\label{nustab}
Let $\bar \Omega \in \Omega^3_+(M)$ be torsion-free. Then for any initial condition sufficiently close to $\bar \Omega$ in the $C^\infty$-topology the solution to \eqref{gDF} for $\nu \in \R^4_+$ exists for all times and converges modulo diffeomorphisms to a torsion-free $\Gt$-structure. 

\end{thm}

\begin{proof}
Let $\Omega \in  \Omega^3_+(M)$ be torsion-free,
i.e.~$d\Omega=\delta_\Omega\Omega=0$. Then
\begin{alignat*}{2}
(D_\Omega Q_0)\dot\Omega &=-\frac{1}{7} \delta_\Omega[d\dot\Omega]_1,\quad &
(D_\Omega Q_1)\dot \Omega &= - \frac{1}{36}  \delta_\Omega[d\dot\Omega]_7 \\
(D_\Omega Q_2) \dot \Omega &=  - p_\Omega(d [\delta_\Omega p_\Omega \dot
\Omega ]_{14}),\quad & (D_\Omega Q_3) \dot \Omega &=  -\delta_\Omega [ d \dot \Omega ]_{27}
\end{alignat*}
and
\ben
(D_\Omega \Lambda)(\dot \Omega) = -3 d[\delta_\Omega\dot\Omega]_7.
\ee
We set $L_{\nu}:=D_\Omega \widetilde
Q_{\nu}$ and $L:=D_\Omega
\widetilde Q$ as in \cite{wewi10}. Then we get
\ben
L_{\nu} = -\nu_0 \frac{1}{7}
\delta_\Omega[d\dot\Omega]_1 - \nu_1 \frac{1}{36}
\delta_\Omega[d\dot\Omega]_7 - \nu_2 p_\Omega(d [\delta_\Omega p_\Omega \dot
\Omega ]_{14})-\nu_3\delta_\Omega [ d \dot \Omega ]_{27} -3\varepsilon(\nu)d[\delta_\Omega\dot\Omega]_7
\ee
and hence
\ben
\langle - L_{\nu} \dot \Omega, \dot \Omega \rangle_{L^2_\Omega} \geq
\varepsilon(\nu) \langle -L \dot \Omega, \dot \Omega \rangle_{L^2_\Omega}
\quad \forall \dot \Omega \in \Omega^3(M)
\ee
with $\varepsilon(\nu)=\min_{i=0,1,2,3}\nu_i/36$ as above. In particular, $L_{\nu}$ is non-positive and the G\r arding inequality holds. The proof then proceeds along the same lines as the one given in~\cite{wewi10} for the Dirichlet energy flow.
\end{proof}
%
%
%
\section{$\Gt$-solitons}
%
\subsection{Symmetries}
%
Recall that one has a natural $\Diff(M)_+$-action on $\Omega_+^3(M)$ given by pullback and that $\mathcal D$ is $\Diff(M)_+$-invariant, i.e.\ $\mathcal D(\varphi^* \Omega) = \mathcal D(\Omega)$ for all $\varphi \in \Diff(M)_+$. This implies that
\be\label{equiv}
\varphi^* Q(\Omega) = Q(\varphi^*\Omega).
\ee
Further, any symmetry of the initial condition $\Omega_0$ is preserved by the Dirichlet energy flow:

\begin{lem}\label{invariance}
Let $(\Omega_t)_{t \in [0,T)}$ be a solution to \eqref{DF} with initial condition $\Omega_0$. If $\varphi^* \Omega_0 =\Omega_0$ for some $\varphi \in \Diff(M)_+$, then $\varphi^*\Omega_t = \Omega_t$ for all $t \in [0,T)$.
\end{lem}

\begin{proof}
Using equation \eqref{equiv} one gets that $(\varphi^*\Omega_t)_{t \in [0,T)}$ is a solution to \eqref{DF} with initial condition $\varphi^*\Omega_0$. Since $\varphi^*\Omega_0 = \Omega_0$, uniqueness of the Dirichlet energy flow implies that $\varphi^*\Omega_t = \Omega_t$ for all $t \in [0,T)$.
\end{proof}

Secondly, one has a natural $\R_+$-action on $\Omega_+^3(M)$ given by scaling with respect to which $\mathcal D$ is positively homogeneous, i.e.
\begin{equation}\label{D-homog}
\mathcal D(\lambda \Omega) = \lambda^{\frac{5}{3}}\mathcal D(\Omega).
\end{equation}
for all $\lambda \in \R_+$.

\begin{lem}\label{Q-homog}
One has $Q(\lambda \Omega) = \lambda^{\frac{1}{3}} Q(\Omega)$ for all $\lambda \in \R_+$.
\end{lem}

\begin{proof}
Using equation \eqref{D-homog} we calculate
\begin{align*}
D_{\lambda \Omega} \mathcal D (\dot \Omega) &= \frac{d}{dt}\Bigr|_{t=0}\mathcal D(\lambda \Omega + t \dot \Omega)\\
&= \lambda^{\frac{5}{3}} \frac{d}{dt}\Bigr|_{t=0}\mathcal D(\Omega + t \lambda^{-1} \dot \Omega) \\
&= \lambda^{\frac{5}{3}} D_\Omega \mathcal D(\lambda^{-1} \dot \Omega) = \lambda^{\frac{2}{3}} D_\Omega \mathcal D(\dot \Omega). 
\end{align*}
Hence 
\ben
D_{\lambda \Omega} \mathcal D(\dot \Omega) = \lambda^{\frac{2}{3}} D_{\Omega}\mathcal D(\dot \Omega) = \lambda^{\frac{2}{3}} \int_M g_\Omega(\grad \mathcal D(\Omega), \dot\Omega) \vol_\Omega
\ee
and on the other hand
\ben
D_{\lambda \Omega} \mathcal D(\dot \Omega) = \int_M g_{\lambda\Omega}(\grad \mathcal D(\lambda\Omega), \dot\Omega) \vol_{\lambda\Omega} = \lambda^{\frac{1}{3}} \int_M g_\Omega(\grad \mathcal D(\lambda\Omega), \dot\Omega) \vol_\Omega.
\ee
Here we have used the fact that $\vol_{\lambda\Omega} = \lambda^{\frac{7}{3}} \vol_\Omega$ and $g_{\lambda \Omega} = \lambda^{-2} g_\Omega$ on 3-forms. Comparing these two expressions we get the result.
\end{proof}

\begin{rmk}
As a consequence of the preceding lemma, if $\Omega_t$ is a solution to \eqref{DF} on $[0,T)$ and $\lambda>0$, then the space-time rescaling
$
\Omega^\lambda_t : = \lambda \Omega_{\lambda^{-2/3} t}
$
is again a solution to \eqref{DF}, defined on $[0,\lambda^{2/3}T)$.
\end{rmk}
%
\subsection{A Bianchi-type identity}
For some fixed background $\Gt$-structure $\Omega$, consider the operator
\ben
\lambda_{\Omega}^*: \mathcal X(M) \rightarrow \Omega^3(M)\,,\, X \mapsto \mathcal{L}_X\Omega
\ee
and its formal adjoint with respect to $L^2_{g_{\Omega}}$, namely
\ben
\lambda_{\Omega}: \Omega^3(M) \rightarrow \mathcal X(M)\,,\, \dot \Omega \mapsto -X_{\Omega}(\dot \Omega) - \dot \Omega \llcorner d \Omega,
\ee
where $X_{\Omega}(\dot\Omega)=-\delta_{\Omega}\dot\Omega\llcorner\Omega$. As usual we identify 1-forms and vector fields using $g_{\Omega}$. 
Recall that we have an $L^2$-orthogonal decomposition
\be\label{tsdecomp}
\Omega^3(M)=\ker\lambda_{\Omega}\oplus\im\lambda_{\Omega}^*,
\ee
where the second summand is tangent to the $\Diff(M)_+$-orbit through $\Omega$, see Proposition 5.6 and Lemma 7.3 in \cite{wewi10}.

\begin{lem}\label{Bianchi}
For all $\Omega\in\Omega^3_+(M)$ we have $\lambda_\Omega (Q(\Omega)) = 0$ and $\lambda_\Omega\Omega = 0$.
\end{lem}

\begin{proof}
The proof proceeds along the same lines as Kazdan's derivation of the usual Bianchi identity in \cite{ka81}: If $\mathcal{F}: \Omega^3_+(M) \rightarrow \R$ is a $\Diff(M)_+$-invariant functional, then  $\lambda_\Omega (\grad \mathcal F(\Omega))=0$, since the level-set $\mathcal F^{-1}(\mathcal F(\Omega))$ contains the $\Diff(M)_+$-orbit through $\Omega$. Now by definition, $Q(\Omega)=-\grad \mathcal D(\Omega)$, which yields $\lambda_\Omega(Q(\Omega))=0$. Secondly, from equation~\eqref{phider} it follows that 
\ben
\grad\mc H(\Omega) = \tfrac{1}{3}\Omega
\ee
which gives $\lambda_\Omega\Omega=0$.
\end{proof}

\begin{rmk} 
The equation $\lambda_\Omega \Omega = 0$ is equivalent to $\tau_1 = \tilde\tau_1$, where in the the definition of the torsion forms one has 
\ben
d \Omega = \tau_0 \star_\Omega \Omega + 3 \tau_1 \wedge \Omega + \star_\Omega \tau_3
\ee
and
\ben
d \star_\Omega\Omega = 4 \tilde\tau_1 \wedge \star_\Omega\Omega + \tau_2 \wedge \Omega
\ee
for $\tilde \tau_1$ a priori different from $\tau_1$. Indeed, $\lambda_\Omega \Omega = (\delta_\Omega\Omega) \llcorner \Omega - \Omega \llcorner d\Omega=0$ is equivalent to
\be\label{Bianchieq}
([\delta_\Omega\Omega]_7) \llcorner \Omega = \Omega \llcorner ([d\Omega]_7).
\ee
Substituting $[\delta_\Omega\Omega]_7 = -4 \star_\Omega \tilde\tau_1 \wedge \star_\Omega\Omega$ and $[d\Omega]_7 = 3 \tau_1 \wedge \Omega$ we obtain that equation \eqref{Bianchieq} is equivalent to
\be\label{Bianchieq2}
-4 \star_\Omega (\tilde\tau_1 \wedge \star_\Omega \Omega) \llcorner \Omega = 3 \Omega \llcorner (\tau_1 \wedge \Omega).
\ee
A routine calculation establishes for $\xi \in \Omega^1(M)$ the identities $\Omega \llcorner (\xi \wedge \Omega) = -4 \xi$ and $\star_\Omega(\xi \wedge \star_\Omega\Omega) \llcorner \Omega = 3 \xi$. Hence the left-hand side of equation \eqref{Bianchieq2} equals $-12 \tilde\tau_1$, whereas the right-hand side equals $-12 \tau_1$.
\end{rmk}

\begin{cor}\label{einstein}
If $\Omega \in \Omega^3_+(M)$ satisfies $Q(\Omega) = f \cdot \Omega$ for $f \in C^\infty(M)$, then $f$ is constant, i.e.\ $Q(\Omega) = \lambda \Omega$ for $\lambda \in \R$.
\end{cor}

\begin{proof}
Applying $\lambda_\Omega$ to the equation $Q(\Omega) = f \cdot \Omega$ yields the equation $\lambda_\Omega(f\Omega) = 0$ using Lemma \ref{Bianchi}. On the other hand
\begin{align*}
\lambda_\Omega (f \Omega) &= -\delta_\Omega(f \Omega) \llcorner \Omega - f \Omega \llcorner d\Omega\\
&= (df \llcorner \Omega - f \delta_\Omega \Omega) \llcorner \Omega - f \Omega \llcorner d\Omega\\
& = (df \llcorner \Omega) \llcorner \Omega - f \lambda_\Omega\Omega = (df \llcorner \Omega) \llcorner \Omega,
\end{align*}
where we have again used Lemma~\ref{Bianchi} in the last line. Now since $ (\xi \llcorner \Omega) \llcorner \Omega = 3 \xi$ for all $\xi \in \Omega^1(M)$ we conclude that $df = 0$, i.e.~$f$ is constant.
\end{proof}

Next we consider the operator $\widetilde Q_{\bar\Omega}(\Omega)=Q(\Omega)+\lambda^*_{\Omega}(X_{\bar\Omega}(\Omega))$, $\Omega,\bar\Omega\in\Omega^3_+(M)$, defined in~\cite{wewi10}.

\begin{cor}\label{slice}
If $\Omega \in \Omega^3_+(M)$ satisfies $\widetilde Q_{\bar\Omega}(\Omega) = 0$, then $Q(\Omega)=0$, i.e.\ $\Omega$ is torsion-free.
\end{cor}

\begin{proof}
Applying $\lambda_\Omega$ to the equation
\be\label{tildeQ}
\widetilde Q_{\bar\Omega}(\Omega)=Q(\Omega) + \lambda_\Omega^*(X_{\bar \Omega}(\Omega)) = 0
\ee
yields the equation $\lambda_\Omega\lambda_\Omega^*(X_{\bar \Omega}(\Omega)) = 0$ using Lemma \ref{Bianchi}. Hence $\lambda_\Omega^*(X_{\bar \Omega}(\Omega)) = 0$ and therefore $Q(\Omega)=0$.
\end{proof}

\begin{rmk}
Note that if $M$ has finite fundamental group or more generally satisfies $H^1(M,\R)=\{0\}$, then $\widetilde Q_{\bar\Omega}(\Omega)=0$ also implies $X_{\bar\Omega}(\Omega)=0$. Indeed, since $Q(\Omega)=0$, $\Omega$ is torsion-free and $\mc{L}_{X_{\bar\Omega}(\Omega)}\Omega=0$. Hence, $g_\Omega$ is Ricci-flat and $X_{\bar\Omega}(\Omega)$ is Killing. But this implies that $X_{\bar\Omega}(\Omega)$ is parallel and therefore its dual 1-form is harmonic. In general, a parallel Killing vector field has no zeros unless it is identically vanishing. Hence the dual of $X_{\bar\Omega}(\Omega)$ is a closed, nowhere vanishing $1$-form. By Tischler's theorem~\cite{ti70}, $M$ must globally fibre over the circle. Note however that non-trivial parallel Killing vector fields can exist: If $X$ is a Calabi-Yau threefold, then the product $X\times S^1$ admits a natural torsion-free $\Gt$-structure for which the coordinate vector field $\partial_t$ on $S^1$ is a parallel Killing vector field. Conversely, by standard holonomy theory (cf.\ for instance~\cite{be87}), a torsion-free $\Gt$-manifold $(M,\Omega)$ with non-trivial parallel Killing vector field is reducible, that is {\em locally} of the form $X\times S^1$ for $X$ a Calabi-Yau manifold.
\end{rmk}
%
\subsection{The soliton equation}
%
\begin{dfn}
A triple $(\Omega_0,X_0,\mu_0)$ with $\Omega_0 \in \Omega^3_+(M)$, $X_0 \in \mathcal{X}(M)$ a vector field and $\mu_0 \in \R$, which satisfy the equation
\ben
Q(\Omega_0) = \mu_0 \Omega_0 + \mathcal L_{X_0} \Omega_0
\ee
is called a $\Gt$-{\em soliton structure}. A solution to \eqref{DF} of the form
\ben
\Omega_t = \mu(t)\varphi_t^* \Omega_0
\ee
for some function $\mu(t)$ and a family of orientation-preserving diffeomorphisms $\varphi_t$ is called a $\Gt$-{\em soliton solution}. 
\end{dfn}

A particular case of a soliton structure is a $\Gt$-structure $\Omega_0$ satisfying the equation $Q(\Omega_0) = \mu_0 \cdot \Omega_0$ for some constant $\mu_0 \in \R$. The ansatz
\ben
\Omega_t = \mu(t) \Omega_0\,,\, \mu(0)=1
\ee
yields using Lemma \ref{Q-homog}
\begin{align*}
\partial_t \Omega_t &= \mu'(t) \Omega_0\\
Q(\Omega_t) &= \mu(t)^{\frac{1}{3}} \mu_0 \Omega_0
\end{align*}
and hence the ODE
\begin{equation}\label{ODE}
\mu'(t) = \mu_0 \mu(t)^{\frac{1}{3}}\,,\,\mu(0)=1.
\end{equation}
The solution of \eqref{ODE} is given by 
\ben
\mu(t) = \Bigl( \frac{2\mu_0}{3}t + 1 \Bigr)^{\frac{3}{2}}
\ee
on some maximal time interval $[0,T_{max})$. As in the Ricci-flow case one has more generally:

\begin{lem}
Let $(\Omega_0,X_0,\mu_0)$ be a $\Gt$-soliton structure. Then
\begin{equation}\label{solution}
\Omega_t:=\mu(t)\varphi_t^*\Omega_0 
\end{equation}
is a $\Gt$-soliton solution on $[0,T_{max})$ for $\mu(t) = ( \frac{2\mu_0}{3}t + 1 )^{\frac{3}{2}}$ and $\varphi_t$ the flow of the time-dependent vector field $\mu(t)^{-\frac{2}{3}}X_0$. The associated metric flow is given by 
\ben
g_t = \mu(t)^{\frac{2}{3}} \varphi_t^* g_0.
\ee
Conversely, if $\Omega_t = \mu(t)\varphi_t^*\Omega_0$ is a $\Gt$-soliton solution on $[0,T_{max})$, then $(\Omega_0,X_0,\mu_0)$ with $X_0=\frac{d}{dt}\bigr|_{t=0}\varphi_t$ and $\mu_0=\mu(0)$ is a $\Gt$-soliton structure.
\end{lem}

\begin{proof}
Differentiating equation \eqref{solution} we get
\begin{align*}
\partial_t \Omega_t &= \varphi_t^* \bigl( \mu(t)^{\frac{1}{3}} \mathcal L_{X_0}(\Omega_0) + \mu'(t) \Omega_0\bigr)\\
Q(\Omega_t) &= \varphi_t^* \mu(t)^{\frac{1}{3}} Q(\Omega_0)
\end{align*}
which yields the claim upon substituting \eqref{ODE}. The evolution of the associated metric $g_t$ immediately follows from its scaling behaviour.   
\end{proof}

\begin{rmk} 
By the preceding lemma, a $\Gt$-soliton structure and a $\Gt$-soliton solution are essentially the same thing. We will therefore simply refer to both the $\Gt$-soliton structure or the corresponding soliton solution as a $\Gt$-{\em soliton}.
\end{rmk}

\begin{dfn}
A $\Gt$-soliton $(\Omega_0,X_0,\mu_0)$ is called {\em expanding}, if $\mu_0>0$; {\em steady}, if $\mu_0=0$; and {\em shrinking}, if $\mu_0<0$. It is called {\em trivial} if $Q(\Omega_0) = \mu_0\Omega_0$.
\end{dfn}

Using this terminology we can state the following:

\begin{prp}\label{no_expanders}
Let $(\Omega_0,X_0,\mu_0)$ be a $\Gt$-soliton. Then the following holds:

{\rm (i)} Any $\Gt$-soliton $(\Omega_0, X_0, \mu_0)$ is trivial, i.e.~already satisfies $Q(\Omega_0) = \mu_0\Omega_0$.  

{\rm (ii)} One has $\mu_0 \leq 0$, i.e.\ there are no expanding $\Gt$-solitons.

{\rm (iii)} If $\Omega_t$ denotes the corresponding soliton solution, then $T_{max} = \infty$ in the steady case and $T_{max}=-\frac{3}{2\mu_0}$ in the shrinking case.
\end{prp}

\begin{proof}
To prove the first assertion we apply $\lambda_{\Omega_0}$ to the equation
\ben
Q(\Omega_0) = \mu_0\Omega_0 + \mathcal L_{X_0}\Omega_0 =  \mu_0\Omega_0 +\lambda_{\Omega_0}^*X_0.
\ee
This gives, using Lemma \ref{Bianchi}, the equation $\lambda_{\Omega_0}\lambda_{\Omega_0}^*X_0=0$, hence $\mathcal L_{X_0}\Omega_0=0$.

Secondly, for $\mu_0>0$ we would have
\ben
\frac{d}{dt} \mathcal{D}(\Omega_t) = \frac{d}{dt} \mathcal{D}(\mu(t)\Omega_0) = \frac{5}{3}\mu_0\mu(t)\mathcal{D}(\Omega_0)>0
\ee
which is incompatible with the monotonicity of $\mathcal D$. The remaining statements follow from the behaviour of the solution of the ODE \eqref{ODE}.
\end{proof}
\begin{rmk}
For a shrinking soliton one clearly has $\lim_{t\rightarrow T_{max}}\mu(t)=0$ and therefore $\lim_{t\rightarrow T_{max}}\mc{H}(\Omega_t)=\lim_{t\rightarrow T_{max}}\mc{D}(\Omega_t)=0$. This follows easily from the scaling behaviour of these functionals.
\end{rmk}
%
\subsection{A constrained variational principle}
%
Next we ask for critical points of $\mc{D}$ under the constraint $\mc{H}(\Omega)=1$. 
Let $\Omega^3_{+,1}(M)$ be the submanifold of $\Omega^3_+(M)$ consisting of positive $3$-forms of total volume $1$. Its tangent space at $\Omega$ is $\ker D_\Omega\mc{H}$. Now by~\eqref{phider}, $\dot{\mc{H}}_\Omega=\langle\dot\Omega,\Omega\rangle/3$ so that $T_\Omega\Omega^3_{+,1}(M)=\Omega^\perp$, the $3$-forms which are perpendicular to $\Omega$ \wrt the natural $L^2$-product. On the other hand, we need $\grad\mc{D}=-Q$ to be orthogonal to $T_\Omega\Omega^3_{+,1}(M)$, hence a constrained critical point $\Omega$ satisfies $Q(\Omega)=\mu_0\Omega$ for some constant $\mu_0\in\R$. In view of Proposition~\ref{no_expanders} we obtain an alternative characterisation of $\Gt$-solitons.

\begin{cor}
A positive $3$-form $\Omega$ is a $\Gt$-soliton \iff $\Omega$ is a critical point of $\mc{D}$ subject to $\mc{H}\equiv1$.
\end{cor} 

\begin{rmk}
The results of this section apply mutatis mutandis to the generalised Dirichlet energy functionals $\mc{D}_\nu$, $\nu \in \R^4_+$. More precisely, we say that $(\Omega_0, X_0,\mu_0)$ is a {\em $\mc{D}_\nu$-soliton} if the equation $Q_\nu(\Omega_0) = \mu_0 \Omega_0 + \mc{L}_{X_0}\Omega_0$ holds. Since $\mc{D}_\nu$ shares the same symmetries with $\mc{D}$, we obtain the Bianchi identity $\lambda_\Omega(Q_\nu(\Omega))=0$. Hence we may deduce that any $\mc{D}_\nu$-soliton is trivial with $\mu_0 \leq 0$. The explicit solution to the soliton equation remains unchanged.
\end{rmk}
%
%
%
\section{Examples}
%
%
\subsection{Homogeneous spaces}
%
Consider a compact homogeneous space $M=G/H$. Then $G$ acts on $M$ via diffeomorphisms coming from left translations. Let $\mf{g}=\mf{h}\oplus\mf{m}$ be the decomposition at Lie algebra level from the inclusion $H\hookrightarrow G$, where $\mf{m}$ is some complement invariant under the isotropy action of $H$ (the adjoint action of $G$ restricted to $H$). The space of $G$-invariant $\Gt$-forms is precisely the space of $H$-invariant $\Gt$-forms in $\Lambda^3\mf{m}^\ast$. Since invariant critical points can be obtained by restricting the functional to invariant $\Gt$-forms, we are left with a finite-dimensional variational problem. We will illustrate this procedure for the Dirichlet energy functional $\mc D$.

\medskip

\paragraph{\em The round sphere.}
We think of $S^7$ as the homogeneous space $\spin(7)/\Gt$. Then $\mf{spin}(7)=\Lambda^2\R^{7*}=\mf{g}_2\oplus\mf{m}$ by~\eqref{pqdecomp}, where $\mf{m}$ is isomorphic to the $7$-dimensional irreducible vector representation of $\Gt$. Hence $\Lambda^3\mf{m}^\ast\cong\mb{1}\oplus\mf{m}\oplus\odot^2_0\mf{m}$ (also cf.\ our first convention at the end of Section~\ref{introduction}) is a decomposition into irreducible $\Gt$-modules, and we find a one-dimensional space of $\spin(7)$-invariant $\Gt$-forms spanned by $\Omega_0$. In fact, if we think of $S^7$ as the unit octonians with induced metric $g_0$ (the round metric), then at $p\in S^7$, $\Omega_{0,p}(u,v,w)=g_{0,p}\big(p,u\cdot(\bar v\cdot w)-w\cdot(\bar v\cdot u)\big)$ (here $\,\bar{}\,$ and $\cdot$ denote conjugation and multiplication on $\Oc$). Since $Q(\Omega_0)$ must be also $\spin(7)$-invariant by Lemma~\ref{invariance}, we deduce $Q(\Omega_0)=c\Omega_0$ for some nonpositive constant $c$. Furthermore, $H^3(S^7;\R)=0$ so that $\Omega_0$ cannot be torsionfree, whence $Q(\Omega_0)\not=0$.

\medskip

\paragraph{\em The squashed sphere.}
Now consider $S^7$ as the homogeneous space $G/H=\Sp(2)\times\Sp(1)/\Sp(1)\times\Sp(1)$ defined by the embedding
\ben
(a,b)\in\Sp(1)\times\Sp(1)\mapsto(\left(\begin{array}{cc}a&0\\0&b\end{array}\right)\!,b).
\ee
The complex irreducible representations of $\Sp(1)\cong\SU(2)$ are obtained from the symmetric powers $\sigma_p=\odot^p\C^2$ of the standard vector representation on $\C^2$. Endowed with some negative multiple of the Killing form $G/H$ becomes a normal Riemannian homogeneous space (cf.~Definition 7.86 in \cite{be87}) with orthogonal decomposition $\mf{g}=\mf{h}\oplus\mf{m}$. As an $\Sp(1)\times\Sp(1)$-space, $\mf{m}=\mb{1}\otimes\sigma_2\oplus\sigma_1\otimes\sigma_1=:\mf{m}'\oplus\mf{m}''$. Here, by abuse of notation, $\sigma_1\otimes\sigma_1$ (which is of real type) also denotes  the underlying real representation. In the resulting decomposition of $\Lambda^3\mf{m}^*$, we find two trivial representations, namely $\Lambda^3\mf{m}'^*\cong\R$ and one in $\mf{m}'^*\otimes\Lambda^2\mf{m}''^*$ (cf.~\cite{alse11}). If $f_1$, $f_2$ and $f_3$ denotes an orthonormal basis of $\mf{m}'$, the first one is spanned by\footnote{Here and in the sequel, $f^{123}$ will be shorthand for $f^1\wedge f^2\wedge f^3$.} $\Omega_1=f^{123}$. For the second invariant form $\Omega_2$ we note that $\Lambda^2\mf{m}''^\ast=\mb{1}\otimes\sigma_2\oplus\sigma_2\otimes\mb{1}$ which is just the decomposition into self- and antiselfdual forms. Consequently, if $e_1,\ldots,e_4$ is an orthonormal basis for $\mf{m}''$, then $\Omega_1=\sum_k f^k\wedge\omega_k$ where
\ben
\omega_1=e^{12}+e^{34},\quad\omega_2=e^{13}-e^{24},\quad\omega_3=e^{14}+e^{23}.
\ee
The $G$-invariant forms
\ben
\mc I=\{\Omega_{a,b}:=-a^3\Omega_1+ab^2\Omega_2\,|\,a,\,b>0\}
\ee
are of $\Gt$-type and compatible with the natural orientation. To compute the $G$-invariant critical points we must compute $\mc D$ on $\mc I$. We first note that $\Omega_{a,b}$ induces the metric $g_{a,b}=-a^2B|_{\mf m'}-b^2B|_{\mf m''}$ so that $\vol_{a,b}=a^3b^4e^{1234}\wedge f^{123}$ and 
\ben
\star_{a,b}\Omega_{a,b}=-b^4e^{1234}+a^2b^2(f^{23}\wedge\omega_1-f^{13}\wedge\omega_2+f^{12}\wedge\omega_3).
\ee 
We compute the commutators $[\cdot\,,\cdot]_{\mf m}$ and thus the exterior differentials of $e_1,\ldots,f_3$. Upon suitably rescaling $B$ we find
\ben
d\Omega_{a,b}=12ab^2e^{1234}+(10ab^2+2a^3)(-f^{23}\wedge\omega_1+f^{13}\wedge\omega_2-f^{12}\wedge\omega_3)
\ee
and $d\star_{a,b}\Omega_{a,b}=0$. Consequently, $|d\Omega_{a,b}|^2=24(7a^2b^{-4}+25a^{-2}+10b^{-2})$, whence
\ben
\mc{D}(\Omega_{a,b})=12(7a^5+10a^3b^2+25ab^4)\mr{Vol},
\ee
with $\mr{Vol}$ the total volume of $G/H$ with respect to $\vol_{1,1}=e^{1234}\wedge f^{123}$. Subject to the constraint $a^3b^4=1$ the critical point equations read
\ben
7a^4+6a^2b^2+5b^4=3\mu a^2b^4,\quad a^2+5b^2=\mu a^2 b^2,\quad a^3b^4=1
\ee
for some constant $\tau$. Substituting $u=a^2$ and $v=b^2$ shows that $u=v$ and $\mu=6/v$. Hence $a=1$, $b=1$ and $\mu=6$ is the unique solution which gives the soliton $\Omega_{1,1}$. The resulting metric is the so-called {\em squashed} metric. 
%
\subsection{Nearly parallel $\Gt$-structures}
The previous two examples define in fact {\em nearly parallel $\Gt$-structures} (see for instance~\cite{fkms97}). These were first investigated by Gray \cite{gr71} (who called them weak holonomy $\Gt$-structures). This is a $\Gt$-structure given by a $\Gt$-form $\Omega$ satisfying
\ben
d\Omega=\tau_0\star_\Omega\Omega
\ee
for some constant $\tau_0\not=0$. In particular, $d\star_\Omega\Omega=0$ so that alternatively, we may characterise nearly parallel $\Gt$-structures as those for which all torsion forms but $\tau_0$ do vanish. By abuse of language, we refer to such an $\Omega$ itself as a nearly parallel $\Gt$-structure. The associated metric is necessarily Einstein with positive constant scalar curvature $s_\Omega=\tfrac{21}{7}\tau_0^2$.

\begin{thm}\label{weakg2sol}
If $\Omega$ is a nearly parallel $\Gt$-structure, then
\be\label{nearly-parallel}
Q_\nu(\Omega)=-\tfrac{5}{42}\nu_0\tau_0^2(\Omega)\Omega
\ee
for all $\nu=(\nu_0,\nu_1,\nu_2,\nu_3) \in \R^4_+$. In particular, $\Omega$ is a $\Gt$-soliton.
\end{thm}
\begin{proof}
First we note that $D_\Omega\mc{D}_k(\dot\Omega)=\int_M\dot\tau_{k,\Omega}\wedge\star_\Omega\tau_k(\Omega)+\tfrac{1}{2}\int_M\tau_k(\Omega)\wedge\dot\star_\Omega\tau_k(\Omega)$. But for a nearly parallel $\Gt$-form $\Omega$ we have $\tau_k=0$, $k\not=0$, so that $\grad\mc{D}_k(\Omega)=0$ and in particular $Q_\nu(Q)=-\nu_0\grad\mc D_0(\Omega)$. We contend that for general $\Omega\in\Omega^3_+(M)$,
\be\label{d0grad}
\grad\mc{D}_0(\Omega)=-\tfrac{1}{6}\tau_0^2\Omega+\tfrac{2}{7}\tau_0\star_\Omega d\Omega+\tfrac{1}{7}\star_\Omega(d\tau_0\wedge\Omega).
\ee
If this is true, then $\grad\mc D_0(\Omega)=\tfrac{5}{42}c^2\Omega$ for nearly parallel $\Omega$, whence the result. It remains to show~\eqref{d0grad}. We first determine $\dot\star_\Omega$, the derivative of the map $\Lambda^3_+\to\Hom(\Lambda^0,\Lambda^7)$ which sends $\Omega$ to $\star_\Omega$. As this is a pointwise computation we can write $\dot\Omega=\dot A^*\Omega$, where $\dot A=\dot A_0$ for a smooth curve $A_t\subset\GL(7)$ with $A_0=\Id$. Then
\ben
\dot\star_\Omega=\tfrac{d}{dt}\bigr|_{t=0}\star_{A^*_t\Omega}=\tfrac{d}{dt}\bigr|_{t=0}A^*_t\star_\Omega A^{-1*}_t=\dot A^*\star_\Omega,
\ee
for $\GL(7)$ acts trivially on $0$-forms. In general, if $v,w\in\Lambda^1$, the action is given by $(v\otimes w)^*\alpha^p=v\wedge(w\llcorner\alpha^p)$ for $\alpha^p\in\Lambda^p$. Using the standard formul{\ae} $\star_\Omega(v\llcorner\alpha^p)=(-1)^{p+1}v\wedge\star_\Omega\alpha^p$ and $\star_\Omega(v\wedge\alpha^p)=(-1)^pv\llcorner\star_\Omega\alpha^p$ we get
\ben
\dot A^*\star_\Omega=\mr{Tr}(\dot A)\star_{\Omega}-\star_\Omega(\dot A^t)^*=\mr{Tr}(\dot A)\star_\Omega.
\ee
On the other hand, we have $\dot A^* \Omega=\dot A^*_1 \Omega+\dot A^*_7 \Omega+\dot A^*_{27}\Omega$ where we used the decomposition of $\dot A \in\Lambda^1\otimes\Lambda^1$ given by~\eqref{endodecomp}. Since $\dot A_1=\tfrac{3}{7}\mr{Tr}(\dot A)\id$ we have
\be\label{a1comp}
\dot A_1^*\Omega=\tfrac{3}{7}\mr{Tr}(\dot A)\Omega.
\ee
Hence
\ben
\dot\star_\Omega\tau_0=\tau_0\mr{Tr}(\dot A)\star_\Omega 1=\tfrac{1}{7}\tau_0\mr{Tr}(\dot A)\Omega\wedge\star_\Omega\Omega=\tfrac{1}{3}\tau_0\dot\Omega\wedge\star_\Omega\Omega.
\ee
To compute the linearisation of $\tau_0(\Omega)=\star_\Omega(d\Omega\wedge\Omega)/7$ we note that $\star_\Omega^2=\id$ implies $\star_\Omega\dot\star_\Omega=-\dot\star_\Omega\star_\Omega$, whence
\begin{align*}
\dot\tau_{0,\Omega} & = \dot\star_\Omega\big(\star_\Omega\tau_0(\Omega)\big)+\tfrac{1}{7}\star_\Omega(d\dot\Omega\wedge\Omega+\dot\Omega\wedge d\Omega)\\
& = -\tfrac{1}{3}\tau_0(\Omega)\star_\Omega(\dot\Omega\wedge\star_\Omega\Omega)+\tfrac{1}{7}\star_\Omega(d\dot\Omega\wedge\Omega+\dot\Omega\wedge d\Omega).
\end{align*}
From
\ben
\langle\grad\mc D_0(\Omega),\dot\Omega\rangle_\Omega=\int_M\tau_0\dot\tau_{0\Omega}\vol_\Omega+\tfrac{1}{6}\int_M\tau_0^2\dot\Omega\wedge\star_\Omega\Omega
\ee
equation~\eqref{d0grad} easily follows.
\end{proof}

\begin{rmk}
The factor appearing in the soliton equation~\eqref{nearly-parallel} can also be computed using the homogeneity of $\mc D_\nu$: If $d\Omega=\tau_0\star_\Omega\Omega$, then by Euler's rule
\ben
\langle Q_\nu(\Omega), \Omega \rangle_\Omega=- D_\Omega\mc D_\nu(\Omega)=-\tfrac{5}{3} \mc D_\nu(\Omega)=-\tfrac{5}{42}\nu_0 \tau_0^2\langle\Omega,\Omega\rangle_\Omega.
\ee
In particular it follows that
\be\label{tau0exp}
\tau_0^2(\Omega)=\frac{2}{\nu_0}\cdot\frac{\mc D(\Omega)}{\mc H(\Omega)}.
\ee
\end{rmk}

\begin{cor}
Let $\Omega \in \Omega^3_+(M)$ be torsion-free. Then there exists a neighbourhood of $\Omega$ in $\Omega^3_+(M)$ with respect to the $C^\infty$-topology which does not contain any shrinking $\mc D_\nu$-solitons, and in particular no nearly parallel $\Gt$-structures.
\end{cor}

\begin{proof}
Choose a neighbourhood $\mathcal U \subset \Omega^3_+(M)$ such that for any initial condition $\Omega_0 \in \mathcal U$ the conclusion of Theorem~\ref{stab} holds. Now if $\Omega_0$ were a shrinking $\mc D_\nu$-soliton, then $T_{max}<\infty$ according to Proposition \ref{no_expanders}, which is impossible. 
\end{proof}

\begin{rmk}
The previous corollary should be compared with Theorem 1.2 in~\cite{dww05} which asserts that a Ricci-flat metric which admits nonzero parallel spinors (as it is the case for $g_\Omega$ with $\Omega$ torsion-free) cannot be smoothly deformed into a metric of positive scalar curvature.
\end{rmk}
%
%
%
\section{Soliton deformations}
Let $\bar \Omega \in \Omega^3_+(M)$ be a fixed nearly parallel $\Gt$-structure, i.e.~$d \bar \Omega = \bar\tau_0 \star_{\bar \Omega} \bar \Omega$ for some constant $\bar \tau_0 \neq 0$. In this final section we linearise the $\Gt$-soliton equation
\be\label{solitoneq}
S_{\bar \Omega}(\Omega):=Q(\Omega)+\tfrac{5}{6}\bar\tau_0^2\Omega=0
\ee
at $\bar \Omega$ and study the premoduli space of $\Gt$-soliton deformations.
%
\subsection{The linearised soliton equation}
In order to linearise the $\Gt$-soliton equation we need a lemma first. Recall the map
\ben
\Theta:\Omega^3_+(M)\to\Omega^4(M),\quad\Omega\mapsto\star_\Omega\Omega
\ee
from Convention (ii) in Section~\ref{introduction}. Its linearisation at $\Omega$ is given by $\dot\Theta_\Omega=\star_\Omega p_\Omega(\dot\Omega)$ where $p_\Omega(\dot\Omega):=\tfrac{4}{3}[\dot\Omega]_1+[\dot\Omega]_7-[\dot\Omega]_{27}$.

\begin{lem}\label{corddotTheta}
Let $\Omega\in\Omega^3_+(M)$. For $x\in M$ let $\Omega_t = A_t^*\Omega_x$ for a curve $A_t\subset\GL(7)$ such that $A_0=\Id_{T_xM}$. If we define $s_\Omega (\dot \Omega) := [\dot \Omega]_1 - [\dot \Omega]_7 + [\dot \Omega]_{27}$, then for the second derivative $\ddot\Theta_\Omega:=\frac{d^2}{dt^2}\bigr|_{t=0}\Theta(\Omega_t) $ at $x$ we find
\begin{align*}
\ddot \Theta_\Omega =& \frac{1}{3} g (\Omega, \dot \Omega) \star_\Omega(p_\Omega - s_\Omega)\dot \Omega + 2 \star_\Omega(\dot A^t)^{*2}\Omega - \star_\Omega s_\Omega \ddot \Omega\\
&+ \frac{1}{3}\big(g(\ddot \Omega,\Omega) -  g(s_\Omega \dot \Omega, \dot \Omega)\big) \star_\Omega \Omega.
\end{align*}
In particular we have
\ben
\ddot \Theta_\Omega = \frac{1}{3} g (\Omega, \dot \Omega) \star_\Omega(p_\Omega - s_\Omega)\dot \Omega + 2 \star_\Omega(\dot A^t)^{*2}\Omega - \frac{1}{3} g(s_\Omega \dot \Omega, \dot \Omega) \star_\Omega \Omega.
\ee
for $\ddot \Omega = 0$.
\end{lem}
\begin{proof}
Writing $A_t = A_t A_{t_0}^{-1} A_{t_0}$ we get
\ben
\frac{d}{dt}\Bigr|_{t=t_0}A_t^* \Theta(\Omega)= A_{t_0}^* \frac{d}{dt}\Bigr|_{t=t_0}(A_t A_{t_0}^{-1})^*\Theta(\Omega) = A_{t_0}^* (\dot A_{t_0}A_{t_0}^{-1})^*\Theta(\Omega)
\ee
and hence
\ben
\frac{d^2}{dt^2}\Bigr|_{t=t_0}\Theta(\Omega_t)=\big((\dot A^*)^2 + \ddot A^* - (\dot A^2)^* \big)\Theta(\Omega).
\ee
In the same way we obtain
\be\label{omegadotdot}
\ddot \Omega = ( (\dot A^*)^2 + \ddot A^* - (\dot A^2)^*)\Omega.
\ee
Now
\begin{align*}
(\dot A^*)^2 \Theta(\Omega) &= \dot A^* ( \dot A^* \star_\Omega\Omega)\\
&= \dot A^* ( \tr \dot A \star_\Omega \Omega - \star_\Omega (\dot A^t)^*\Omega)\\
&=\tr \dot A \bigl( \tr \dot A \star_\Omega \Omega - \star_\Omega (\dot A^t)^*\Omega)\bigr) - \tr \dot A \star_\Omega (\dot A^t)^*\Omega + \star_\Omega(\dot A^t)^{*2} \Omega\\
&=\tr \dot A \star_\Omega (p_\Omega-s_\Omega) \dot \Omega + \star_\Omega(\dot A^t)^{*2} \Omega,
\end{align*}
where we have used $\tr \dot A \star_\Omega \Omega - \star_\Omega (\dot A^t)^*\Omega= \dot \Theta_\Omega$ and $(\dot A^t)^*\Omega = s_\Omega \dot \Omega$. Similarly,
\ben
\ddot A^*\Theta(\Omega) = \ddot A^* \star_\Omega\Omega = \tr \ddot A \star_\Omega \Omega - \star_\Omega (\ddot A^t)^* \Omega
\ee
and
\ben
- (\dot A^2)^* \Theta(\Omega) = -\tr \dot A^2 \star_\Omega \Omega + \star_\Omega(\dot A^2)^{t*}\Omega.
\ee
Finally, using~\eqref{omegadotdot}
\begin{align*}
 &\big((\dot A^*)^2 + \ddot A^* - (\dot A^2)^* \big)\Theta(\Omega) \\
 =&  \tr \dot A \star_\Omega (p_\Omega-s_\Omega) \dot \Omega  + (\tr \ddot A - \tr \dot A^2) \star_\Omega \Omega + 2\star_\Omega ((\dot A^t)^*)^2\Omega-\star_\Omega s_\Omega \ddot \Omega.
\end{align*}
Next we need to compute the expression $\tr(\ddot A - \dot A^2)$. By~\eqref{a1comp} $\tr\dot A = \tfrac{1}{3} g(\Omega, \dot \Omega)$ and similarly $\tr \ddot A = \tfrac{1}{3} g(\Omega, \ddot A^*\Omega)$. Write $(\ddot A - \dot A^2)^* \Omega = \ddot \Omega - (\dot A^*)^2 \Omega$. Then
\ben
\tr (\ddot A - \dot A^2) = \tfrac{1}{3} g( \Omega, (\ddot A - \dot A^2)^* \Omega) = \tfrac{1}{3} g (\Omega, \ddot \Omega - (\dot A^*)^2 \Omega).
\ee
Furthermore, 
\ben
\begin{array}{l}
\,[\dot A^*_{1} \dot \Omega]_1= \tfrac{1}{7} g (\dot A^*_{1} \dot \Omega, \Omega) \Omega = \tfrac{1}{7} g (\dot \Omega, \dot A^*_{1} \Omega) \Omega = \frac{1}{7} | [\dot \Omega]_{1} |^2\Omega\\[5pt]
\,[\dot A^*_7 \dot \Omega]_1 = \tfrac{1}{7} g (\dot A^*_7 \dot \Omega, \Omega) \Omega = - \tfrac{1}{7} g (\dot \Omega, \dot A^*_7 \Omega) \Omega = - \frac{1}{7} |[\dot \Omega]_7 |^2\Omega\\[5pt]
\,[\dot A^*_{27} \dot \Omega]_1= \tfrac{1}{7} g (\dot A^*_{27} \dot \Omega, \Omega) \Omega = \tfrac{1}{7} g (\dot \Omega, \dot A^*_{27} \Omega) \Omega = \frac{1}{7} | [\dot \Omega]_{27} |^2\Omega.
\end{array}
\ee
Hence
\begin{align*}
[\dot A^* \dot A^* \Omega ]_1 &= [\dot A^* \dot \Omega]_1\\
&= [\dot A^*_{1} \dot \Omega]_1 + [\dot A^*_7 \dot \Omega]_1 + [\dot A^*_{27} \dot \Omega]_1\\
&= \tfrac{1}{7} ( |[\dot \Omega]_1|^2 - |[\dot \Omega]_7|^2 + |[\dot \Omega]_{27}|^2) \Omega\\
&=\tfrac{1}{7}g( s_\Omega \dot \Omega, \dot \Omega)\Omega
\end{align*}
and in turn
\ben
\tr (\ddot A - \dot A^2) = \tfrac{1}{3} g (\Omega, \ddot \Omega) - \tfrac{1}{3} g (s_\Omega \dot \Omega, \dot \Omega),
\ee
which yields the assertion.
\end{proof}

\begin{prp}\label{secondvar}
Let $\Omega \in \Omega^3_+(M)$ be a nearly parallel $\Gt$-structure and define $r_\Omega(\dot\Omega):= (\id - p _\Omega)(\dot \Omega)$. Then
\begin{align*}
D_\Omega Q(\dot \Omega) =& -\delta_\Omega d \dot\Omega - p_\Omega d \delta_\Omega p_\Omega \dot \Omega - \tau_0(\star_\Omega d r_\Omega + r_\Omega\star_\Omega d) \dot \Omega\\
 &+ \tau_0^2 ( \tfrac{1}{18} [\dot \Omega]_1 + \tfrac{1}{6} [\dot \Omega]_7 -\tfrac{23}{6} [\dot \Omega]_{27})\\
=&-p_\Omega d(p_\Omega d)^*\dot\Omega-(\star_\Omega d+\tau_0r_\Omega)^2\dot\Omega+\tfrac{1}{6}\tau_0^2\dot\Omega
\end{align*}
for $\tau_0=\tau_0(\Omega)$ and $\dot \Omega \in \Omega^3(M)$.
\end{prp}
\begin{proof}
We compute the linearisation by starting from equations \eqref{qop} and \eqref{quadratic}. First, 
\begin{align*}
D_\Omega(\Omega \mapsto \delta_\Omega d \Omega)(\dot \Omega) &= \dot\star_\Omega d \star_\Omega d \Omega + \star_\Omega d \dot\star_\Omega d \Omega + \star_\Omega d \star_\Omega d \dot \Omega\\
&= \tau_0^2 \dot\star_\Omega \star_\Omega \Omega + \tau_0 \star_\Omega d \dot\star_\Omega \star_\Omega \Omega+ \star_\Omega d \star_\Omega d \dot \Omega\\
&=\tau_0^2 r_\Omega \dot \Omega + \tau_0 \star_\Omega d r_\Omega \dot \Omega + \delta_\Omega d \dot \Omega.
\end{align*}
Second,
\begin{align*}
D_\Omega(\Omega \mapsto p_\Omega d \delta_\Omega\Omega)(\dot \Omega) &= - \dot p_\Omega (d \star_\Omega d \star_\Omega\!\Omega) - p_\Omega(d \dot \star_\Omega d \star_\Omega\Omega) - p_\Omega(d \star_\Omega d \dot\Theta_\Omega)\\
&=p_\Omega d \delta_\Omega p_\Omega \dot \Omega.
\end{align*}
Third we note that $q_\Omega(\nabla^\Omega) = q_\Omega(d \Omega) + q_\Omega(\delta_\Omega \Omega)$, where $q_\Omega(d\Omega)$ and $q_\Omega(\delta_\Omega\Omega)$ are determined by the identities
\be\label{qdom}
q_\Omega(d\Omega) \wedge \star_\Omega \Omega'=\tfrac{1}{2}(\star_\Omega'd \Omega)\wedge d\Omega
\ee
and
\be\label{qdeltom}
q_\Omega(\delta_\Omega\Omega) \wedge \star_\Omega \Omega'=\tfrac{1}{2}(\star_\Omega' d \star_\Omega \Omega)\wedge d \star_\Omega \Omega
\ee
(with $\star'_\Omega=D_\Omega(\Omega\mapsto\star_\Omega)(\Omega')$) valid for all $\Omega' \in \Omega^3(M)$. It follows that $q_\Omega(d\Omega) = - \tfrac{1}{6}\tau_0^2 \Omega$. Indeed, the left hand side of~\eqref{qdom} is twice $\tau_0^2\star'_\Omega\Theta(\Omega)\wedge\Theta(\Omega)$. Now $\Omega=\star_\Omega\Theta(\Omega)$ so that $\Omega'=\star'_\Omega\Theta(\Omega)+\star_\Omega\Theta'_\Omega$. Hence, $[\star'_\Omega\Theta(\Omega)]_1=-[\Omega']_1/3$ which is the only component which survives wedging by $\Theta(\Omega)$. Differentiating equation~\eqref{qdom} therefore implies
\begin{align*}
&D_\Omega(\Omega \mapsto q_\Omega(d\Omega))(\dot \Omega)\wedge \star_\Omega \Omega'\\ 
=& \tfrac{1}{2} (D_\Omega^2 \star)(\dot \Omega, \Omega') d\Omega \wedge d \Omega + \star_\Omega' d \Omega \wedge d \dot \Omega - q_\Omega(d\Omega) \wedge \dot\star_\Omega \Omega'\\
=& \tfrac{1}{2} \tau_0^2 (D_\Omega^2 \star)(\dot \Omega, \Omega') \star_\Omega\Omega \wedge \star_\Omega \Omega + \tau_0 \star_\Omega' \star_\Omega \Omega \wedge d \dot \Omega - \tfrac{1}{6} \tau_0^2 r_\Omega \dot \Omega \wedge \star_\Omega \Omega'.
\end{align*}
On the other hand, differentiating the equation $\Omega = \star_\Omega \Theta(\Omega)$ gives
$\ddot \Omega = \ddot \star_\Omega \Theta(\Omega) + 2 \dot \star_\Omega \dot \Theta_\Omega + \star_\Omega \ddot \Theta_\Omega$. Without loss of generality we may assume that $\Omega_t = (1+t) \Omega$, so in particular $\ddot \Omega = 0$ and hence $\ddot \star_\Omega \Theta_\Omega = -2 \dot \star_\Omega - \star_\Omega \ddot \Theta_\Omega$. From Lemma~\ref{corddotTheta} we deduce
\begin{align*}
&\tfrac{1}{2} \tau_0^2 (D_\Omega^2 \star)(\dot \Omega, \Omega') \star_\Omega\Omega \wedge \star_\Omega \Omega\\
 =& \tau_0^2 \bigl( - \dot \star_\Omega \dot\Theta_\Omega-\tfrac{1}{6} g_\Omega(\Omega,\dot \Omega)(p_\Omega-s_\Omega)\dot \Omega - \big((\dot A^t)^{*2} \Omega + \tfrac{1}{6} g_\Omega(s_\Omega\dot\Omega,\dot\Omega)\Omega\big) \wedge \star_\Omega \Omega'\bigr).
\end{align*} 
Furthermore, the identities
\ben
\begin{array}{l}
-((\dot A^t)^*)^2 \Omega \wedge \star_\Omega \Omega = - (\dot A^t)^*\Omega \wedge \star_\Omega \dot A^*\Omega = - s_\Omega \dot \Omega \wedge \star_\Omega \dot \Omega\\[5pt] -\dot\star_\Omega\dot\Theta_\Omega \wedge \star_\Omega\Omega = - r_\Omega p_\Omega \dot \Omega \wedge \star_\Omega \dot \Omega\\[5pt]
-\tfrac{1}{6} g_\Omega(\Omega, \dot \Omega)(p_\Omega-s_\Omega)\dot \Omega \wedge \star_\Omega \Omega = -\tfrac{7}{18} [\dot \Omega]_1 \wedge \star_\Omega \dot \Omega\\[5pt]
\tfrac{1}{6}g_\Omega (s_\Omega\dot\Omega,\dot \Omega)\Omega \wedge \star_\Omega \Omega = \tfrac{7}{6} s_\Omega \dot \Omega \wedge \star_\Omega \dot \Omega
\end{array}
\ee
imply
\ben
\tfrac{1}{2} \tau_0^2 (D_\Omega^2 \star)(\dot \Omega, \Omega') \star_\Omega\Omega \wedge \star_\Omega \Omega\\
= \tau_0^2 (\tfrac{2}{9} [\dot \Omega]_1 -\tfrac{1}{6} [\dot \Omega]_7 + \tfrac{13}{6} [\dot \Omega]_{27}) \wedge \star_\Omega \Omega'.
\ee
Hence, using
\ben
\tau_0\star_\Omega' \star_\Omega \wedge d \dot \Omega = - \tau_0\star_\Omega \star_\Omega' \Omega \wedge d \dot \Omega = \tau_0r_\Omega \Omega' \wedge d \dot \Omega = \tau_0r_\Omega \star_\Omega d \dot \Omega \wedge \star_\Omega \Omega'
\ee
we arrive at
\begin{align*}
&D_\Omega(\Omega \mapsto q_\Omega(d\Omega))(\dot \Omega)\wedge \star_\Omega \Omega'\\ 
=&  \tau_0^2 (\tfrac{2}{9} [\dot \Omega]_1 -\tfrac{1}{6} [\dot \Omega]_7 + \tfrac{13}{6} [\dot \Omega]_{27}) \wedge \star_\Omega \Omega'
+ \tau_0r_\Omega \star_\Omega d \dot \Omega \wedge \star_\Omega\Omega' - \tfrac{1}{6} \tau_0^2 r_\Omega \dot \Omega\wedge \star_\Omega \Omega'\\
=& \tau_0^2(\tfrac{5}{18}[\dot\Omega]_1 - \tfrac{1}{6} [\dot \Omega]_7 +\tfrac{11}{6}[\dot\Omega]_{27}) \wedge \star_\Omega \Omega'+\tau_0r_\Omega \star_\Omega d \dot \Omega \wedge \star_\Omega\Omega'.
\end{align*}
Similarly, differentiating equation \eqref{qdeltom} we get
\begin{align*}
&D_\Omega(\Omega \mapsto q_\Omega(\delta_\Omega\Omega))(\dot \Omega)\wedge \star_\Omega \Omega'\\ 
=& \tfrac{1}{2} (D_\Omega^2 \star)(\dot \Omega, \Omega') d \Theta(\Omega)\wedge d \Theta(\Omega) + \star_\Omega' d \Theta(\Omega) \wedge d \dot \Theta_\Omega - q_\Omega(\delta_\Omega\Omega) \wedge \dot\star_\Omega \Omega'\\
=&0,
\end{align*}
for $d\Theta(\Omega) = q_\Omega(\delta_\Omega\Omega) =0$. Hence,
\begin{align*}
D_\Omega(\Omega \mapsto q_\Omega(\nabla^\Omega))(\dot\Omega) &= D_\Omega(\Omega \mapsto q_\Omega(d\Omega))(\dot \Omega)\\
&=\tau_0r_\Omega \star_\Omega d \dot \Omega + \tau_0^2(\tfrac{5}{18}[\dot\Omega]_1 - \tfrac{1}{6} [\dot \Omega]_7 +\tfrac{11}{6}[\dot\Omega]_{27}). 
\end{align*}
Summing up we obtain
\begin{align*}
(D_\Omega Q)(\dot\Omega) =&  - \delta_\Omega d \dot \Omega - p_\Omega d \delta_\Omega p_\Omega \dot \Omega - \tau_0\star_\Omega d r_\Omega \dot \Omega - \tau_0^2 r_\Omega \dot \Omega\\
& -\tau_0r_\Omega \star_\Omega d \dot \Omega - \tau_0^2(\tfrac{5}{18}[\dot\Omega]_1 - \tfrac{1}{6} [\dot \Omega]_7 +\tfrac{11}{6}[\dot\Omega]_{27})\\ 
=&  - \delta_\Omega d \dot \Omega - p_\Omega d \delta_\Omega p_\Omega \dot \Omega -  \tau_0 (\star_\Omega d r_\Omega + r_\Omega \star_\Omega d)\dot \Omega\\
& + \tau_0^2 ( \tfrac{1}{18} [ \dot \Omega]_1 + \tfrac{1}{6} [\dot\Omega]_7 - \tfrac{23}{6} [\dot \Omega]_{27}),
\end{align*}
which is the desired result.
\end{proof}

\begin{rmk}
In particular, we see that $D_\Omega Q(\Omega)=-\tfrac{5}{18}\tau_0^2\Omega$ which, of course, follows directly from differentiating $Q((1+t)\Omega)=(1+t)^{1/3}Q(\Omega)$ at $t=0$ (cf.\ Lemma~\ref{Q-homog}). 
\end{rmk}

As a corollary to Proposition~\ref{secondvar}, we immediately obtain the linearisation of the operator $S_{\bar \Omega}$ at $\bar\Omega$:

\begin{cor}
Let $\bar \Omega \in \Omega^3_+(M)$ be a nearly parallel $\Gt$-structure. Then
\ben
D_{\bar \Omega} S_{\bar\Omega}(\dot \Omega)=-p_{\bar \Omega}d(p_{\bar \Omega}d)^*\dot\Omega-(\star_{\bar \Omega} d + \bar\tau_0 r_{\bar \Omega})^2\dot\Omega+\bar\tau_0^2\dot\Omega
\ee
for $\bar\tau_0=\tau_0(\bar\Omega)$ and $\dot \Omega \in \Omega^3(M)$.
\end{cor}
%
\subsection{The premoduli space}
As above, let $\bar \Omega \in \Omega^3_+(M)$ be a fixed nearly parallel $\Gt$-structure on $M$. We wish to study the space of $\Gt$-soliton deformations of $\bar \Omega$, i.e.~solutions $\Omega \in \Omega^3_+(M)$ to the soliton equation \eqref{solitoneq} close to $\bar \Omega$ modulo the action of diffeomorphisms.
Towards that end, we first investigate the linear equation $D_{\bar\Omega}S_{\bar \Omega}(\dot\Omega)=0$. As this parallels the corresponding theory for the Einstein premoduli space as developed by Koiso, we follow~\cite{be87} and~\cite{bo80} and only sketch the main points. Recall the $L^2$-orthogonal decomposition 
\ben
\Omega^3(M) = \im\lambda_{\bar \Omega}^\ast\oplus \ker \lambda_{\bar \Omega}.
\ee
given in~\eqref{tsdecomp}. By Ebin's slice theorem~\cite{eb68}, $\ker \lambda_{\bar \Omega} = T_{\bar\Omega} \mc S_{\bar \Omega}$ integrates to a slice $\mc S_{\bar \Omega}$ for the $\diff_0(M)$-action. Hence the space $\sigma(\bar\Omega)$ of {\em infinitesimal soliton deformations} of $\bar\Omega$ consists of $\dot\Omega\in\Omega^3(M)$ satisfying the equations
\ben
D_{\bar \Omega}S_{\bar\Omega}(\dot \Omega) =0 \quad \text{and} \quad \lambda_{\bar \Omega}(\dot \Omega) = 0.
\ee
The {\em premoduli space} $\mc M(\bar\Omega)$ of $\Gt$-soliton deformations at $\bar\Omega$ is the set of $\Gt$-solitons in the slice $\mc S_{\bar \Omega}$ near $\bar\Omega$. To investigate the structure of $\sigma(\bar\Omega)$ and $\mc M(\bar\Omega)$ further we introduce the linear operator
\ben
P_{\bar \Omega}: \Omega^3(M)\to\Omega^3(M)\,,\; P_{\bar \Omega}(\dot\Omega):= D_{\bar \Omega} S_{\bar\Omega}(\dot \Omega)-\lambda_{\bar\Omega}^*\lambda_{\bar \Omega}(\dot\Omega),
\ee
which is clearly symmetric.

\begin{lem}
The operator $P_{\bar \Omega}$ is elliptic. 
\end{lem}
\begin{proof}
The operator $P_{\bar\Omega}$ differs from the linearisation of the Dirichlet-DeTurck operator only in the lower order terms, cf.~in particular equation (32) in \cite{wewi10}. Hence it has the same symbol and the claim follows from Lemma 5.7 in \cite{wewi10}.
\end{proof}

Since any infinitesimal soliton deformation of $\bar\Omega$ lies in the kernel of $P_{\bar\Omega}$ we immediately conclude from ellipticity:

\begin{cor}
The space $\sigma(\bar\Omega)$ is finite dimensional.
\end{cor}

To discuss the structure of the premoduli space we first prove the following

\begin{lem}
The restricted linear operator $D_{\bar\Omega}S_{\bar \Omega}:T_{\bar\Omega}\mc S_{\bar \Omega}\to\Omega^3(M)$ has closed\footnote{Here and thereafter, this refers to the natural extension of $D_{\bar\Omega}S_{\bar\Omega}$ to Sobolev- or H\"older-spaces.} image.
\end{lem}
\begin{proof}
Clearly, $P_{\bar \Omega}(T_{\bar\Omega}\mc S_{\bar\Omega})=D_{\bar\Omega}S_{\bar\Omega}(T_{\bar\Omega}\mc S_{\bar \Omega})$. As an elliptic operator, $P_{\bar \Omega}$ has closed image. Furthermore, $\lambda_{\bar\Omega}\circ P_{\bar\Omega}=\lambda_{\bar\Omega}\lambda_{\bar\Omega}^\ast\circ\lambda_{\bar\Omega}$ and thus
\ben
P_{\bar\Omega}(T_{\bar\Omega}\mc S_{\bar\Omega})
\subset P_{\bar\Omega}(\Omega^3(M))\cap\ker\lambda_{\bar\Omega}
\subset P_{\bar\Omega}\big(\lambda_{\bar\Omega}^{-1}(\ker\lambda_{\bar\Omega}\lambda_{\bar\Omega}^\ast)\big).
\ee
Now $L_{\bar \Omega} :=\lambda_{\bar\Omega}\lambda_{\bar \Omega}^\ast$ is elliptic. Indeed, for the principal symbol applied to a covector $\xi\in T^*_xM$ we find that $\sigma(L_{\bar\Omega})(x,\xi)v=i(v\otimes\xi)^\ast\bar\Omega$. This is injective, for $(v\otimes\xi)^\ast\bar\Omega=0$ implies $v\otimes\xi\in\Lambda^2\subset\Lambda^1\otimes\Lambda^1$ on representation theoretic grounds, that is, $v\otimes\xi$ is skew. But this is impossible for a decomposable endomorphism unless $v=0$. Hence $g_{\bar \Omega}(\sigma(L_{\bar\Omega})(x,\xi)v,v)=-|\sigma(\lambda_{\bar\Omega}^*)(x,\xi)v|_{\bar\Omega}^2$ is negative-definite. Consequently, $\ker L_{\bar\Omega}$ is finite-dimensional and so $T_{\bar\Omega}\mc S_{\bar\Omega}$ is of finite codimension in $\lambda_{\bar\Omega}^{-1}(\ker\lambda_{\bar\Omega}\lambda_{\bar\Omega}^\ast)$. Since $T_{\bar\Omega}\mc S_{\bar\Omega}$ is also closed, $P_{\bar\Omega}(T_{\bar\Omega}\mc S_{\bar\Omega})$ is closed in $P_{\bar\Omega}\big(\lambda_{\bar\Omega}^{-1}(\ker\lambda_{\bar\Omega}\lambda_{\bar\Omega}^\ast)\big)$. As a result, $P_{\bar\Omega}(T_{\bar\Omega}\mc S_{\bar\Omega})$ is closed in $P_{\bar\Omega}(\Omega^3(M))\cap\ker\lambda_{\bar\Omega}$ and thus in $\Omega^3(M)$.
\end{proof}

Let $p:\Omega^3(M)\to D_{\bar\Omega}S_{\bar\Omega}(T_{\bar\Omega}\mc S_{\bar\Omega})$ be the orthogonal projection. By the previous lemma, $p\circ S_{\bar\Omega}:\mc S_{\bar\Omega}\to D_{\bar\Omega}S_{\bar\Omega}(T_{\bar\Omega}\mc S)$ is a submersion at $\bar\Omega$. It is also a real analytic map, since $g_{\bar\Omega}$ is Einstein (hence real analytic in harmonic coordinates, cf.~\cite{kade81}) and $\Delta_{g_{\bar\Omega}}\bar\Omega=\bar\tau_0^2\bar\Omega$ (so that $\bar\Omega$ is real analytic as a solution of an elliptic PDE with real analytic coefficients). As a consequence, $Z:=p\circ S_{\bar\Omega}^{-1}(0)$ is a real analytic submanifold with tangent space $\ker D_{\bar\Omega}S_{\bar\Omega}\cap T_{\bar\Omega}\mc S_{\bar\Omega}=\sigma(\bar\Omega)$. Restricted to $Z$, $S_{\bar\Omega}$ is also real analytic so that $(S_{\bar\Omega}|_Z)^{-1}(0)$, the premoduli space of solitons, is a real analytic subset. We thus arrive at the following conclusion (compare with Koiso's work~\cite{ko83} in the Einstein case).

\begin{thm}\label{premodstructure}
The slice $\mc S_{\bar\Omega}$ contains a finite dimensional real analytic submanifold $Z$ such that $Z$ contains $\mc M(\bar\Omega)$ as a real analytic subset and $T_{\bar\Omega}Z=\sigma(\bar\Omega)$.
\end{thm}

\begin{ex}
Consider the spaces
\begin{align*}
\sigma_1 &=\{\gamma\in\Omega^3_{27}(M)\,|\,\star_{\bar\Omega} d\gamma=-\bar\tau_0\gamma\},\\
\sigma_2 &=\{\gamma\in\Omega^3_{27}(M)\,|\,\star_{\bar\Omega} d\gamma=-3\bar\tau_0\gamma\},\\
\sigma_3 &=\{\gamma\in\Omega^3_{27}(M)\,|\,\star_{\bar\Omega} d\gamma=-3\bar\tau_0^2\gamma\}.
\end{align*}
Any $\gamma\in\sigma_{1,2}$ is coclosed. Since $d\bar\Omega=\tau_0\star_{\bar\Omega}\Omega$, we also have $\gamma\llcorner d\Omega=0$. Furthermore, any $\gamma\in\sigma_3$ is closed, hence $[\delta_{\bar\Omega}\gamma]_7=0$ (see the proofs of Lemma 3.3 and Proposition 5.3 in~\cite{alse11}). Therefore, $\lambda_{\bar\Omega}(\gamma)=0$ in all three cases. It is straightforward to check that $P_{\bar\Omega}\gamma=0$ for $\gamma\in\sigma_{1,2,3}$. By Theorem 6.2 in \cite{alse11} these spaces correspond to the infinitesimal Einstein deformations of $\bar\Omega$. We do not know whether they exhaust all of $\sigma(\bar\Omega)$.
\end{ex}

{\em Acknowledgments.} The authors would like to thank Bernd Ammann and Joel Fine for useful discussions on related matters. Furthermore, they thank the Hausdorff Research Institute for Mathematics at Bonn for hospitality during the preparation of the paper.
%
%
%

\end{document}